\documentclass[french,english]{article}
\usepackage{graphicx,color}
\usepackage{amsmath,epsfig,amssymb,mathrsfs,amsthm}
\usepackage{geometry}
\usepackage[T1]{fontenc}
\usepackage[latin9]{inputenc}
\usepackage{amstext}
\usepackage{esint}
\usepackage{babel}
\usepackage{subfigure}
\usepackage{amsfonts}
\usepackage{mathtools}
\usepackage{epstopdf}
\usepackage{ifthen}
\usepackage{color,soul,dsfont,enumerate}
\def\D{{\mathcal{D}}}
\def\S{{\mathcal{S}}}

\def\T{\mathcal{T}}
\def\CF{{\widehat{\mathscr{P}}}}
\def\r{\textbf{r}}
\def\a{\alpha}
\def\n{\textbf{n}}
\def\e{\textbf{e}}
\def\u{\textbf{u}}
\def\v{\textbf{v}}
\def\eu{{\rm e}}
\def\ju{{\rm j}}
\def\du{{\rm d}}
\def\1{{\mathds{1}}}
\def\Der{\mathrm{D}}
\def\FL{(-\Delta)^{\gamma/2} }
\def\P{\mathscr{P}}

\newcommand{\C}{ \mathbb{C}}
\newcommand{\R}{ \mathbb{R}}

\newcommand{\N}{ \mathbb{N}}

\newtheorem{lemma}{\textbf{Lemma}}
\newtheorem{cor}{\textbf{Corollary}}
\newtheorem{prop}{\textbf{Proposition}}
\newtheorem{theo}{\textbf{Theorem}}
\newtheorem{defn}{\textbf{Definition}}

\title{On the Continuity of Characteristic Functionals \\ and Sparse Stochastic Modeling}
\author{Julien Fageot, Arash Amini, and Michael Unser}

\begin{document}
\maketitle

\paragraph{Abstract.} 

The  characteristic functional is the infinite-dimensional  generalization of {the} Fourier transform for measures on function spaces.
 It characterize{s} the statistical law of {the associated} stochastic process in the same way {as a} characteristic function specif{ies} the probability distribution of {its corresponding} random variable. 
 Our  goal in this work is to lay the foundations of the \emph{innovation model}, a (possibly) non-Gaussian probabilistic model for sparse signals.    
This is achieved by using the characteristic functional to specify sparse stochastic processes {that are} defined as linear transformation{s} of general continuous-domain white L\'evy noises (also called \emph{innovation processes}). 
We prove the existence of a broad class of sparse processes by using the Minlos-Bochner theorem. This requires a careful study of the regularity properties, especially the $L^p$-boundedness,  of the characteristic functional of the innovations.
We are especially interested in the functionals that are only defined for  $p<1$ since they appear to be associated with the sparser kind of processes.
Finally, we apply our main theorem of existence to two specific subclasses of processes with specific invariance properties. 

\paragraph{Keywords.} Characteristic functional - Generalized stochastic process  -  Innovation model - White L\'evy noise.

\paragraph{Mathematics Subject Classification.} 60G20 - 60H40 - 60G18.

\section{Introduction}

\subsection{Presentation of the Innovation Model}

Sparsity plays a significant role in the mathematical modeling of {real-world} signals. 
A signal is said to be \emph{sparse} if its energy tends to be concentrated in few coefficients in some transform-domain.
Natural images are known to have such a sparse representation.
Moreover, numerous statistical studies have shown that typical biomedical and natural images are non-Gaussian~\cite{Simoncelli2003}. 
These empirical facts highlight the fundamental limits of probabilistic models based on Gaussian priors \cite{Mumford2000}.  
{The sparsity-based theories developed for overcoming these limitations include}
wavelets \cite{Mallat1999} (with {powerful} applications in image coding and processing) and, more recently,  compressed sensing \cite{Donoho2004, Candes2006sparse}. They are inherently deterministic. \\

A new general model has been recently developed in order to reconcile the sparsity paradigm of signal processing with a probabilistic formulation. Its general foundations and motivations   were discussed in \cite{Unser_etal2011, Unser_etal2011bis}. The main hypotheses are as follows.
\begin{itemize}
    \item A signal is modeled as a random continuous-domain function $s$ defined on $\R^d$. Hence, $s$ is the stochastic process that captures the statistical properties of the signal.
    
    \item The process $s$ can be linearly decoupled, which implies the existence of  a linear whitening operator $\mathrm{L}$ such that
        \begin{equation}\label{eq:innovation-model}
            \mathrm{L}s=w,
        \end{equation}
        where $w$ is a continuous-domain innovation process, also called white noise, which is \emph{not necessarily Gaussian}. The term ``innovation'' reflects the property that $w$ is the unpredictable component of the process. 
         \end{itemize}
Following the terminology of \cite{Unser_etal2011}, these two hypotheses define the \emph{innovation model} (see Figure 1).

\begin{figure}[h!] \label{IM}
\centering
  \includegraphics[width=0.38 \textwidth]{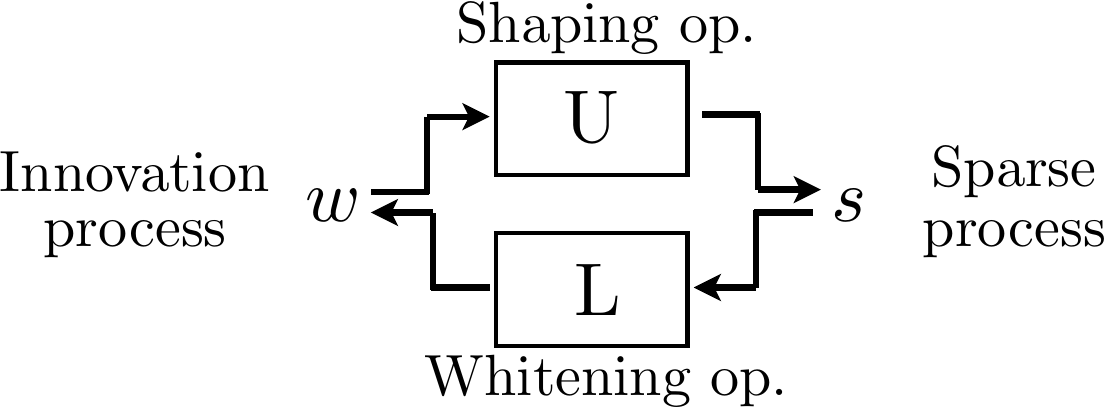} 
\caption{Innovation model}
\end{figure}

The innovation model provides a  mathematical framework that complies with the sparse behavior of real-world signals for at least two theoretical reasons. First, a process $s$ following \eqref{eq:innovation-model}, given a non-Gaussian innovation process $w$,  has been shown to be sparser (\emph{i.e.}, more compressible) than any Gaussian one \cite{Unser_etal2011}. We can therefore refer to these processes as  \emph{sparse processes}. Second, it has been demonstrated for the  case of symmetric $\alpha$-stable (S$\alpha$S) AR(1)  processes that better decoupling is achieved in a wavelet-like representation than with the traditional sine basis or the KLT  \cite{Pad2013wavelets}.  \\
The innovation model has already been applied to various fields of image processing such as  
Bayesian estimation from noisy samples of sparse processes~\cite{Amini2013Bayesian}, 
algorithms for the optimal quadratic estimation of sparse processes~\cite{KamilovMMSE}, 
and reconstruction techniques based on sparse and self-similar processes \cite{Bostan2013SampTA}. It was also found to be useful  in inverse problems, involving MRI, deconvolution, and  X-ray tomography reconstruction problems  \cite{BosKamNil, BostanIsbi2012}. \\
While these examples show that sparse processes are highly relevant for practical applications, the theory currently available   is based on too-constraining assumptions. In particular, it excludes some of the sparsest processes such as S$\alpha$S   with $\alpha <1$, for which wavelets have been found empirically to be optimal \cite{Pad2013wavelets}. More generally,  the compatibility between a linear operator $\mathrm{L}$ and an innovation process $w$, defined as the existence of a process $s$ such that $\mathrm{L}s =w$, is a crucial question that needs to be addressed. \\

The innovation model is formulated within the extended framework of Generalized Stochastic Processes (GSP), the stochastic counterpart of Schwartz theory of generalized functions. 
This probabilistic theory was historically  introduced in the 50's by   Gelfand \cite{Gelfand1955generalized}.
The larger part of the litterature on GSP is concerned with second-order processes with a special emphasis on the Gaussian case. Recent examples include results on  existence and regularity of Gaussian GSP \cite{bel1998linear}, construction of Gaussian and Poisson isotropic and self-similar GSP \cite{Bierme2010}, and classification of Gaussian stationary mean-square increments GSP \cite{Jorgensen2011}. 
Gelfand's formulation extends beyond the second-order family. For instance, the processes with unbounded variance take a particular relevance in the context of sparsity \cite{Aminicompressibility}, \cite{Gribonval2012compressibility}. 
We consequently consider \cite{GelVil4} as the starting point of our own developments.\\
This  framework of generalized stochastic processes enables the definition of the innovation processes, which cannot be defined as classical stochastic processes. Gelfand   defines innovation processes as random generalized functions. As a consequence, the GSP  are not observed pointwise but by forming duality products with test functions $\varphi \in \D$, where $\D$ is the space of smooth and compactly supported functions.
For a fixed $\varphi$, the observation $\langle w,\varphi\rangle$ is then a conventional random variable. 
In addition to that, Gelfand's framework appears to be particularly adapted for the development of the theory related to the innovation model and for its applicability in signal processing. \\

\subsection{Contributions}
Our goal is to define the broadest   framework that guarantees the existence of sparse processes. In that sense, our work can be seen as an extension of the existence results of \cite{Unser_etal2011}.  
We now give  the three main results of this paper. \\

\textbf{a) Proper definition of innovation processes over $\S'$} \\
While the usual definition of an innovation process (\emph{i.e.}, a continuous-domain white noise) is over the space $\D'$, the topological dual of $\D$ \cite{GelVil4}, we therein define innovation process{es}  over the space $\S'$ of tempered generalized functions.  This extension requires to identify a sufficient condition for an innovation process to be tempered, which is the focus of Theorem \ref{theoS}. 
       We show in addition that most of the innovation processes in $\D'$ are  supported on $\S'$ (Theorem \ref{theoDS}). This ensures the compatibility between the two constructions over $\D'$ and $\S'$ when they occur. \\ 
       	The choice of $\S'$ is driven by the desire to make the innovation model applicable to signal processing. Adopting $\S'$ as the definition space   allows us to extend $\langle w ,\varphi\rangle$ to the case of non-compactly supported functions, which are crucial in signal-processing applications.  \\
	     
\textbf{b) Existence of the broader family of sparse processes} \\
 We   investigate the compatibility of pairs $(w,\mathrm{L})$ of innovation processes and linear operators and introduce a large class of valid combinations. 
    Before describing our contributions, we briefly summarize the knowm results on the existence of sparse processes of the form $s=\mathrm{L}^{-1}w$.\\
    Gelfand formulates the general definition of innovation processes on $\D'$. Hence, $(w,\mathrm{Id})$ is a compatible pair for all $w$ defined by \cite{GelVil4}.
    An immediate extension of this result is that $(w,\mathrm{L})$ is also valid if the adjoint operator $\mathrm{L}^{-1}$ has a $\D$-stable inverse operator. In that case, one can directly define $s$ according to
    	\begin{equation} \label{s_def}
    	\langle s,\varphi \rangle = \langle w ,\mathrm{L}^{*-1} \varphi \rangle.
	\end{equation}
	Indeed,  the $\D$-stability of $\mathrm{L}^{*-1}$ ensures that $\langle w,\mathrm{L}^{*-1} \varphi \rangle$ is always well-defined. We  then have $\langle \mathrm{L}s,\varphi \rangle = \langle s, \mathrm{L}^*\varphi \rangle = \langle w , \mathrm{L}^{*-1} \mathrm{L}^{*} \varphi \rangle = \langle w ,\varphi \rangle$ or, equivalently, $\mathrm{L}s=w$.  
	Unfortunately, interesting whitening operators for signal processing do not fullfil this stability condition.
	For instance, in the one-dimension case, the common differential operator $\mathrm{L} =\alpha \mathrm{Id} - \mathrm{D}$ with $\alpha>0$  associated with AR(1) processes is already problematic. Indeed, $\rho_\alpha(t) =  u(t) \mathrm{e}^{-\alpha t}$, with $u(t)$ the Heaviside step function, is the causal Green function of $\mathrm{L}^* = \alpha \mathrm{Id} + \mathrm{D}$. Its inverse $\mathrm{L}^{*-1} \varphi = (\rho*\varphi ) (t)$ is therefore not $\D$-stable. Note  that $\mathrm{L}^{*-1}$ is however  $\S$-stable. This has also encourage us to develop the theory of innovation model over $\S'$  instead of $\D'$.   \\
	As a first step, Unser et al. have expressed the comparability condition over $\S'$. The pair $(w,\mathrm{L})$ is shown to be compatible if there exists $p\geq 1$ such that (i) $\mathrm{L}^{*}$ admits a left-inverse operator $\mathrm{L}^{*-1}$ from $\S$ to $L^p$ and (ii) $w$ is $p$-admissible, a condition that quantifies the level of sparsity of $w$ (see Theorem 3 in \cite{Unser_etal2011}). This theory enables sparse processes to be defined not only for classical differential operators that are typically $\S$-stable, but also for fractional differential operators which require the $L^p$ extension. This generalization is sufficient for most of the cases of practical interest, but has two restrictions. First, it does not encompass the case of high sparsity as it is restricted to $p\geq 1$. Second, the $p$-admissibility condition limits the generality of the results.  \\
	 In this paper,  we formulate a new criterion of compatibility that avoid these restrictions. We show that the generalized stochastic process $s$  over $\S'$ with $\mathrm{L} s = w$ exists if  one can link  the existence of moments for $w$ to the existence of a stable left-inverse for the operator $\mathrm{L}^*$ (Theorem \ref{maintheo}). 
   We present our proofs of sufficiency in the most general setting, which requires us  to  extend  the continuity of the characteristic functional of the innovation processes from $\S$ to $L^p$ spaces (Proposition \ref{mainprop}). \\
         
\textbf{c) Construction of specific subclasses of processes}\\
     We apply our compatibility condition{s} to {two} specific {families} of operators. A  class of self-similar processes is defined   (Proposition \ref{GammaV}) by extending a previous study of the fractional Laplacian operators \cite{Sun-frac}. A class of directional L\'evy processes in dimension $d$ is introduced by the use of directional differential operators (Proposition \ref{existsMondrian}). The latter extends the work done in \cite{Unser_etal2011} for $d=1$.

\subsection{Outline}

This paper is organized as follows. In Section \ref{sec:GSP}, we recall some concepts and results on generalized stochastic processes. In Section \ref{sec:Criterion}, we first present    the general construction of white noises (innovation processes) developed by Gelfand \cite{GelVil4} and adapt it to the space of tempered generalized functions (Section \ref{subsec:LevyNoise}). {Next}, we present and prove  {a criterion for the compatibility of the innovation process $w$ and the whitening operator $\mathrm{L}$ to form a sparse  process (Section \ref{subsec:SparseProcess}).}
The proof {relies on }
continuity bounds for the characteristic functional of innovation processes (Section \ref{subsec:Continuity}). Finally, we apply this criterion in Section \ref{sec:Applications} to two specific classes of operators {and identify }
classes of generalized stochastic processes: {self-similar sparse processes through fractional Laplacian operators (Section \ref{subsec:SelfSimilar}) and directional sparse processes through directional-derivative operators (Section \ref{subsec:Directional}).}

\section{Generalized Stochastic Processes}\label{sec:GSP}

Our main concern  is to define stochastic processes that satisfy the innovation model~\eqref{eq:innovation-model}.  The theory of generalized stochastic processes  is based on functional analysis. {In Table \ref{table:notations}, we provide the definition of function spaces linked to our work.} 
{They include} subspaces of ordinary functions from $\R^d$ to $\R$ (classical functions) {as well as} subspaces of the space $\D'$ of distributions {(also called generalized functions)} \cite{SchDistri}.  

\subsection{Definition of Generalized Stochastic Processes}

We {deviate} from the traditional {time-series approach to} stochastic processes by presenting them as probability measures on a function space $\mathcal{X} $ {of} functions from $E$ to $\R$. Let $\mathcal{X}$ be a topological vector space of real-valued functions. We denote by $\mathcal{A}$ the   $\sigma$-field generated by the cylindric sets. There are the subsets of $E$ defined by $A_{\boldsymbol{x},B}=\{h\in \mathcal{X}, (h(x_1),...,h(x_N)) \in B\}$ for fixed $N\in \mathbb{N}$, {where} $\boldsymbol{x}=(x_1,...,x_N) \in E^N$ and $B$ {is} a Borelian  set in $\R^N$. {For a given probability measure $\mathscr{P}$ on $\mathcal{A}$, }
the \emph{canonical stochastic process} $s$ on $(\mathcal{X},\mathcal{A},\mathscr{P})$ is defined by
\begin{eqnarray*}
 s: & (\mathcal{X},E) & \rightarrow  \mathbb{R} \\
 & (h,x) & \mapsto  h(x) .
\end{eqnarray*}

There are two ways to consider $s$.
\begin{itemize}
\item If $\boldsymbol{x}=(x_1,...,x_N) \in E^N$ is fixed, $h\mapsto (h(x_1),...,h(x_N))$ is a random variable in $\R^N$ with {probability} law $\mathscr{P}_{\boldsymbol{x}} (B) = \mathscr{P} \left( h\in \mathcal{X}, (h(x_1),...,h(x_N)) \in B\right)$ for any  Borelian set {$B$} of $\R^N$. These laws are the \emph{finite-dimensional marginals} of $s$. 

\item For $h{\in \mathcal{X}}$ {following} the probability measure $\mathscr{P}$, {the mapping} $x\mapsto h(x)$ is a sample function of the stochastic process ({\emph{i.e.,}} a random element of $\mathcal{X}$). 
\end{itemize}
\bigskip
 
If $E=\R^d$, we get {back to} the theory of classical (non-generalized) stochastic processes. {The generalized theory of stochastic processes is obtained when $E$ is formed by a set of test functions of $\R^d$ to $\R$. }
Let  $E=\mathcal{T}$ be a locally convex topological vector space (l.c.t.v.s.)---the minimal structure required in functional analysis  \cite{RudinFA}---and {let} $\mathcal{X}=\mathcal{T}'$ {be} the topological dual of $\mathcal{T}${. W}e define the stochastic process $s$ {by} 
\begin{eqnarray}
	 s: & (\mathcal{T}',\mathcal{T}) & \rightarrow  \mathbb{R} \\
 	& (u,\varphi) & \mapsto \langle u,\varphi \rangle . \nonumber
\end{eqnarray}

The random variable $s({\cdot},\varphi)$ is {de}noted {by} $\langle s,\varphi \rangle$. The realization $s(u,.)$,  {which follows} $\mathscr{P}$ {for $u\in \mathcal{T}'$}, is by definition a linear and continuous functional on $\mathcal{T}$. We call such $s$ a \emph{generalized stochastic process} if $\D\subset \mathcal{T} \subset L^2 \subset \mathcal{T}' \subset \D'$, meaning that $s$ {is} a random generalized function. {In \cite{GelVil4}, }Gelfand and Vilenkin {develop} the essential results for $\mathcal{T} = \mathcal{D}$. In this paper, we especially focus on {its} extensions {to} $\mathcal{T} = \S$.

\begin{table}[tbh]
\centering
\caption{Definition of function spaces used in the paper}\label{table:notations}
\bigskip
\begin{tabular}{l | l l l}
\hline
\hline
  Space & Parameter & Definition & Structure\\
  \hline
 $L^p$  & $1\leq p < + \infty$ &  $\lVert f\rVert_p = \left( \int_{\R^d} |f(\textbf{r}) |^p\du\textbf{r} \right)^{1/p} < + \infty $ & Complete normed  \\
 $L^\infty$  &   &  $\lVert f\rVert_\infty = \mathrm{ess} \sup_{\textbf{r} \in \R^d} |f(\r)|  < +\infty$ & Complete normed  \\
 $L^p $  & $0< p < 1$ &  $d_p (f,g) = \int_{\R^d} |(f-g)(\textbf{r}) |^p \du \textbf{r}  < + \infty $ & Complete metric  \\
 $L^{\infty,\a} $  & $0<\alpha<+\infty$ &  $\lVert f\rVert_{\infty,\a} = \lVert  (1+\|\r\|_2^{\a} ) f(\r) \rVert_\infty < + \infty $ & Complete normed  \\
 $\mathcal{R}$ &   & $\bigcap_{\a>0} L^{\infty,\a}$ & Complete metric  \\ 
  $C^\infty$ &   & Infinitely differentiable functions & Vectorial \\ 
 $\D$ &   & $f\in C^\infty$ with compact support & Nuclear  \\
 $\S$ &   & $f \in C^\infty$ with $\partial^{\n} f \in \mathcal{R}$ for all $\n \in \N^d$ & Nuclear  \\
 $\mathcal{O}_M$ &   & Space of slowly increasing functions & Vectorial  \\
 & & s.t. $f \in C^\infty$  with $|\partial^{\n} f (\r)| \leq |P_{\n}(\r)| $ & \\
 & &  for some polynomial $P_{\n}$ and all $\n$  & \\
  \hline
 $\D'$ &   & $u$ linear and continuous functional on $\D$   & Nuclear  \\
 $\S' $&& $u$ linear and continuous functional on $\S$ & Nuclear  \\
 $\mathcal{O}'_C$ && Space of rapidly decreasing generalized functions & Vectorial  \\
 & &   or equivalently   $u \in \S'$ such that $\mathcal{F}u \in \mathcal{O}_M$  & \\
\hline
\hline
\end{tabular}
\end{table}

In Table 1, we list the type of structures that are useful in this paper. All  considered function spaces are vectorial and most of them {are locally convex, as they} have a topological structure inherited from a distance, a norm, or a family of semi-norms \cite{RudinFA}. 
There are essentially two main structures for function spaces: Banach structure for complete normed spaces and nuclear structure for spaces defined by a family of semi-norms {that imply suitable decreasing properties on unit balls \cite{Treves1967}. }
Note that the Banach and nuclear structures are mutually exclusive in infinite dimension. As we shall see, the nuclear structure is  {central} to {the definition of} a continuous-domain  innovation process.

\subsection{The Characteristic Functional}

By analogy with the finite-dimensional case ($\mathcal{T} = \R^N$), where the characteristic function characterizes a probability measure (L\'evy's theorem), we use a Fourier-domain representation to describe the measures on $\mathcal{T}'$.

\begin{defn}
Let $\mathcal{T}$ be a l.c.t.v.s. and let $\mathcal{T}'$ be its topological dual. {The \emph{characteristic functional} of a generalized stochastic process $s$ on $\mathcal{T}'$ associated with the probability measure $\P_s$ is defined as}
\begin{equation}
\widehat{\mathscr{P}}_s(\varphi)=\mathbb{E}\left[\eu^{\ju\langle s,\varphi\rangle}\right]=\int_{\mathcal{T}'}\eu^{\ju\langle u,\varphi\rangle}\du\mathscr{P}_s(u),
\end{equation}
{where $\varphi \in \mathcal{T}$.}
\end{defn}

The characteristic functional contains the definition of all finite-dimensional laws of the process, in particular {the distribution of all} random vectors $X= \left( \langle s,\varphi_1\rangle, \cdots, \langle s,\varphi_N\rangle \right)$. Indeed, the characteristic function of $X$ is {given by}
\begin{equation}
\widehat{p}_X(\boldsymbol{\omega}) = \mathbb{E}[\eu^{\ju\langle \boldsymbol{\omega},X\rangle}] = \widehat{\mathscr{P}}_s(\omega_1\varphi_1+\cdots+\omega_N \varphi_N ) .
\end{equation}
{In Proposition \ref{CFprop}, w}e summarize the main properties of $\widehat{\mathscr{P}}_s$.

\begin{prop} \label{CFprop}
A characteristic functional is {normalized ($\widehat{\P}_s(0)=1$) and is continuous and positive-definite on $\mathcal{T}$. The latter means} that for all $N\in \mathbb{N}$, $a_1,\cdots,a_N \in \C$, and $\varphi_1,..., \varphi_N \in \T$, we have {that}
\begin{equation}
\sum_{i,j} a_i  \overline{a}_j \CF_s(\varphi_i - \varphi_j) \geq 0. 
\end{equation}
\end{prop}

The normalization {property reflects the fact} that $\P_s (\mathcal{T}) = 1$, whereas the positive-definiteness is linked with the {non-negativity} of the measure $\P_s$. Our focus here is on probability measures on the dual space $\mathcal{N}'$ of a nuclear space $\mathcal{N} \subset L^2 \subset \mathcal{N}'$. {The reason is that the converse of Proposition \ref{CFprop} also holds if $\mathcal{T}$ is nuclear (Theorem \ref{theoMB}). }
Notorious examples of nuclear spaces are $\mathcal{D}$, $\mathcal{S}$,  and their duals $\D'$ {(the space of distributions)} and $\mathcal{S}'$ {(the space of tempered distributions), as seen in Table \ref{table:notations}. This} highlights the deep link between nuclear structures and the theory of generalized {processes}. 

\begin{theo}  [Minlos-Bochner] \label{theoMB}
Let $\mathcal{N}$ be a nuclear space and $\CF$ be a continuous, positive-definite functional from $\mathcal{N}$ to $\C$ with $\CF(0)=1$. Then, there exists a unique measure $\mathscr{P}_s$ on $\mathcal{N}'$ such that
$$\widehat{\mathscr{P}}=\CF_s.$$
\end{theo}

Minlos-Bochner's theorem is an extension of Bochner's theorem to the infinite-dimensional setting. It is our key tool to define probability measures on the nuclear space $\mathcal{S}'$.  \\

\section{A Criterion for Existence of Sparse Processes}\label{sec:Criterion}

\subsection{White L\'evy-Schwartz Noise}\label{subsec:LevyNoise}

We first recall some definitions and results from Gelfand and Vilenkin's theory of generalized stochastic processes \cite{GelVil4}, especially the definition of white L\'evy noises on $\D$. 

\subsection*{Innovation Processes on $\D'$}

\begin{defn}\label{def_noise}
A stochastic process $w$ on {$\mathcal{D}'$ characterized by the probability measure $\mathscr{P}_w$} is said to be
\begin{itemize}
\item \emph{with independent value at every point} if the random variables $X_{1}=\langle w,\varphi_{1}\rangle$ and $X_{2}=\langle w,\varphi_{2} \rangle$ are independent whenever $\varphi_{1}, \varphi_{2} \in \D$ have disjoint supports (\emph{i.e.}, if {$\varphi_{1} \varphi_{2}\equiv 0$) and}
\item \emph{stationary} if the shifted process $w(\cdot -\r_0)$ has the same finite-dimensional {marginals as} $w$. 
\end{itemize}
\end{defn}

{The properties in Definition \ref{def_noise} }can be inferred from the characteristic functional of the process. {Specifically, the independence property corresponds to the condition} 
\begin{eqnarray} \label{CF_indep}
\widehat{\mathscr{P}}_w(\varphi_{1}+\varphi_{2})=\widehat{\mathscr{P}}_w(\varphi_{1})\widehat{\mathscr{P}}_w(\varphi_{2})
\end{eqnarray}
 whenever $\varphi_{1}$ and $\varphi_{2}$ have disjoint supports \cite{GelVil4}.  Moreover, {$w$} is stationary iff. it has the same characteristic functional as the process $w(\cdot -\r_0)$ defined by $\langle w(\cdot-\r_0),\varphi\rangle = \langle w , \varphi(\cdot+\r_0)\rangle$, \emph{i.e.} iff. $\forall \varphi \in \D$ and $\mathbf{r}_0 \in \R^d$, 
\begin{eqnarray} \label{CF_statio}
  \widehat{\mathscr{P}}_w (\varphi (\cdot - \mathbf{r}_0)) = \widehat{\mathscr{P}}_w(\varphi).
  \end{eqnarray}

The functional 
\begin{align}\label{eq:char_func_form}
\CF(\varphi)=\exp\left(\int_{\mathbb{R}^{d}}f(\varphi(\mathbf{r}))\du\mathbf{r}\right)
\end{align} with $f(0)=0$ {satisfies the equations \eqref{CF_indep} and \eqref{CF_statio}}. Moreover, Gelfand and Vilenkin give necessary and sufficient conditions on $f$ so that the functional is continuous and positive-definite over $\D$, and hence, defines a valid innovation process $w$. 

\begin{theo}  [Gelfand-Vilenkin] \label{theoGV}
{D}efine $\CF(\varphi)=\exp\left(\int_{\mathbb{R}^{d}}f(\varphi(\mathbf{r}))\du\mathbf{r}\right)$ on $\D$ where $f$ is a continuous function from $\R$ to $\C$ with $f(0)=0$. The following conditions are equivalent:
\begin{itemize}
\item[(i)] There exists a (unique) {probability} measure $\mathscr{P}_w$ on $\D'$ such that $$\widehat{\mathscr{P}} (\varphi) = \CF_w(\varphi).$$
\item[(ii)] The functional  $\CF$ is a continuous, positive-definite, and normalized ($\CF(0)=1$) functional on $\mathcal{D}$.
\item[(iii)] There exist $\mu\in\mathbb{R}$, $\sigma^{2}\in\mathbb{R}^{+}$, and {a L\'evy measure $V$ with} $\int_{\mathbb{R\backslash}\{0\}}\min(1,a^{2})V(\du a)<\infty$ such that
\begin{align}\label{eq:LevyExponent}
f(\omega)=\ju\mu\omega-\frac{\sigma^{2}\omega^{2}}{2}+\int_{\mathbb{R\backslash}(0)}\left(\eu^{\ju a\omega}-1-\ju\omega a \1_{|a|<1}\right)V(\du a).
\end{align}
\end{itemize}
\end {theo}

{A function $f$ that can be represented in the form of \eqref{eq:LevyExponent} is called a \emph{L\'evy exponent}. The function is alternatively characterized by the triplet $(\mu,\sigma^{2},V)$ known as the \emph{L\'evy triplet}.}

\begin{defn}
A \emph{white L\'evy noise}, or {an} \emph{innovation process} on $\D'$ is a generalized stochastic process $w$ with probability measure $\mathscr{P}_w$ on $\mathcal{D}'$ {that is characterized by $\widehat{\mathscr{P}}_w(\varphi) = \exp\left(\int_{\mathbb{R}^{d}}f(\varphi(\mathbf{r}))\du\mathbf{r}\right)$ for some L\'evy exponent $f$.}
{ In addition, t}he functional $F(\varphi)= \log \CF_w (\varphi)$ is called the \emph{generalized L\'evy exponent} associated with $w$. 
\end{defn}

A white L\'evy  noise on $\D'$ has an independent value at every point and is stationary, which justifies the``white noise'' nomenclature. By Theorem \ref{theoGV}, we have a one-to-one correspondence between  L\'evy exponents $f$ and the white L\'evy noises on $\D'$. Interestingly, one has the same one-to-one correspondence between the family of infinite-divisible probability laws and L\'evy exponents. Indeed, $p_X$ is an infinite-divisible pdf if and only if $\widehat{p}_X(\omega) = \eu^{f(\omega)}$ where $f(\cdot)$ is a valid L\'evy exponent \cite{Sato94}. \\

Gelfand and Vilenkin's constructive {result on} the characteristic functional of {an} innovation process on $\mathcal{D}$  {resolves} the central {barrier} of the positive-definiteness requirement in {applying the} Minlos-Bochner theorem. Indeed, {we shall show in Proposition \ref{positdefin} that, for extending} Theorem \ref{theoGV} {to} larger spaces of test functions{, we only require to prove the continuity of the functional \eqref{eq:char_func_form} as} the positive-definiteness is automatically inherited.

\begin{prop}\label{positdefin} 
{Let $\T$ be any of $\mathcal{S}$, $L^p$, or $L^p\cap L^q$ for $p,q>0$. Assume $f$ is  a L\'evy exponent such that the functional $ \widehat{\mathscr{P}} (\varphi)= \exp\left(\int_{\mathbb{R}^{d}}f(\varphi(\mathbf{r}))\du\mathbf{r}\right)$ is well-defined (namely, $f(\varphi(\r)) \in L^1$) and is continuous by the natural topology of $\T$}. Then, $\CF$ is also positive-definite over $\T$.
\end{prop}

Note that the {topological} structure {of} $L^p \cap L^q$ depends on the {relative values of $p$ and $q$ with respect to $1$.} If, for instance, $p<1 \leq q$, then $L^p\cap L^q$ is a metric space {with} $d_p(f,g)  + \lVert f-g \rVert_q$ (see Table \ref{table:notations} or \cite{RudinFA}).

\begin{proof}
{From Theorem \ref{theoGV} w}e know that $\CF$ is {well-}defined, continuous, normalized and positive-definite over $\D$. We then use a density argument to extend the positive-definiteness to $\T$. Indeed, $\mathcal{D}$ is dense in {all possible $\T$ of Proposition \ref{positdefin}.} This {result} is well-known for $\S$ and the $L^p$ spaces {with} $p\geq 1$. {T}he proof for {$L^p$ spaces with} $0<p<1$ and for $L^p \cap L^q$ {spaces is also similar}.

Let $\varphi_1,\dots,\varphi_N \in \T$ and $a_1,\dots,a_N\in \C$. {Since $\mathcal{D}$ is dense in $\T$, t}here exist sequences $(\varphi_n^k)_{1\leq n\leq N,k\in \mathbb{N}}$ of functions in $\D$ such that $ \lim\limits_{k  \leftarrow +\infty} \varphi_n^k = \varphi_n$ for all $n$. 
Then, {by} using the continuity of $\CF$ over $\T$, {we obtain that}
\begin{equation}  \sum_{1\leq i,j\leq N} a_i \overline{a}_j \widehat{\mathscr{P}}(\varphi_i - \varphi_j)  = \lim\limits_{k\rightarrow +\infty} \left( \sum_{1\leq i,j\leq N} a_i \overline{a}_j \widehat{\mathscr{P}}(\varphi_i^k - \varphi_j^k)\right) \geq 0.
\end{equation}

\end{proof}

\subsection*{Innovation Processes on $\S'$} 

{We r}ecall that the Minlos-Bochner theorem is  {valid for any} nuclear space {including} $\S' \subset \D'${, which allows us to} generalize Definition \ref{def_noise} to $\S'$.  Moreover, it is possible to characterize the independence and the stationarity of a generalized process on $\S'$ directly on its characteristic functional in the same way we did for $\D'$ in \eqref{CF_indep} and \eqref{CF_statio}.  {Next, we introduce} a sufficient condition on the L\'evy exponent $f$ (more precisely, on the L\'evy measure) to {extend the notion of innovation process to $\mathcal{S}'$ by applying Theorem \ref{theoMB}.}  We first give some definitions.

\begin{defn}\label{LSmeasure}
Let $\mathscr{M}(\mathbb{R}\backslash\{0\})$ be the set of Radon measures on $\mathbb{R}\backslash\{0\}$. For $V \in \mathscr{M}(\R\backslash \{0\})$ and $k\geq0$, we denote
\begin{eqnarray}
\mu_{k}(V) & = & \int_{\mathbb{R}\backslash\{0\}}|a|^{k}V(\du a),\\
\mu_{k}^0(V)  & = & \int_{0<|a|<1}|a|^{k}V(\du a), \\
 \mu_{k}^{\infty} (V) & = & \int_{|a|\geq1}|a|^{k}V(\du a).
\end{eqnarray}
with $\mu_{k}(V)=\mu_{k}^{0}(V)+\mu_{k}^{\infty}(V)$. {Further, w}e define
\begin{align}
\mathscr{M}(p,q)=\Big\{ V\in\mathscr{M}\left(\mathbb{R}\backslash\{0\}\right)\: \big|\:\mu_{q}^{0}<\infty\:\text{and }\mu_{p}^{\infty}<\infty \Big\} .
\end{align}
{Hence, t}he set of L\'evy measures {corresponds to} $\mathscr{M}(0,2)$. We also define the set of \emph{L\'evy-Schwartz measures} {as a subset of L\'evy measures adapted for extending the framework of Gelfand and Vilenkin to the Schwartz space $\S$ (see Theorem \ref{theoS}) by}
\begin{align}
\mathscr{M}(0^{+},2) \, = \, {\bigcup_{\epsilon>0}}\mathscr{M}(\epsilon,2) ~ \subset\mathscr{M}(0,2) .
\end{align}
\end{defn}

{It is not difficult to check the following properties of the sets $\mathscr{M}(p,q)$:}
\begin{itemize}
\item $\mathscr{M}(p_{1},q_{1})\cap\mathscr{M}(p_{2},q_{2})=\mathscr{M}\big(\max(p_{1},p_{2})\,,\,\min(q_{1},q_{2})\big)$,
\item $\mathscr{M}(p_{1},q_{1})\cup\mathscr{M}(p_{2},q_{2})=\mathscr{M}\big(\min(p_{1},p_{2}) \,,\, \max(q_{1},q_{2})\big)$, 
\item $\mathscr{M}(p_{1},q_{1})\subset\mathscr{M}(p_{2},q_{2})$ $\Leftrightarrow$
$p_{1}\geq p_{2}$ and $q_{1}\leq q_{2}$.
\end{itemize}

{The interest of Definition \ref{LSmeasure} is to focus separately on the behavior} of $V$ around $0$ and at  {infinities}. {It also helps in classifying the innovation processes according to their L\'evy measure. For instance,} Poisson innovations correspond to $V\in \mathscr{M} (0,0)$ {while innovations with} finite variance {are obtained for} $V \in \mathscr{M} (2,2)$. {In Theorem \ref{theoS}, we state our main} result concerning innovation processes over $\mathcal{S}'$.

\begin{theo}  [Tempered innovation processes]  \label{theoS}
Suppose that $f$ is a L\'evy exponent with triplet $(\mu,\sigma^{2},V)${, where $V$ is} a L\'evy-Schwartz measure {(Definition \ref{LSmeasure})}. Then, there exists a unique measure $\mathscr{P}$ on $\mathcal{S}'$ such that 
\begin{align}
\widehat{\mathscr{P}}(\varphi)=\int_{\mathcal{S}'}\eu^{\ju\langle u, \varphi\rangle}\du\mathscr{P}(u)=\exp\left(\int_{\mathbb{R}^{d}}f(\varphi(\mathbf{r}))\du\mathbf{r}\right) {, ~~~ \forall \varphi \in \mathcal{S}.}
\end{align}
The underlying generalized stochastic process $w$ {associated with $\mathscr{P}$ }
is called a tempered innovation process or a white L\'evy-Schwartz  noise. 
\end{theo}

\begin{proof}
{The function space} $\S$ is nuclear, which {justifies the application of} the Minlos-Bochner theorem. Obviously, $\widehat{\mathscr{P}} (0)=1${, and $\widehat{\mathscr{P}}(\varphi)$ is also positive-definite, given it is continuous (Proposition \ref{positdefin}). However, n}ote that it is not {\emph{a priori} evident} that $f(\varphi(\r))$ is {even} integrable for $\varphi \in \S${, whereas the integrability is easily understood for $\varphi \in \D$, since $f$ is continuous and $\varphi$ is of finite support.} 
{We prove Theorem \ref{theoS} by successively establishing the integrability of $f(\varphi(\r))$ for $\varphi \in \mathcal{S}$ and the continuity of the functional $\CF$ on $\mathcal{S}$.}

The proof is based on a control {on} the generalized L\'evy exponent developed  in Section \ref{subsec:Continuity}. To use this result we first remark that $\S$ is a subspace of all $L^p$ spaces. Moreover, the continuity {of a functional over} $\S$ implies {its} continuity over {any} $L^p$ space {with} $p>0$. 

{Since} $V$ is a L\'evy-Schwartz measure, there exists $0<\epsilon\leq 1$ such that $V \in \mathscr{M}(\epsilon,2)$. {U}sing Corollary \ref{corgconti} (Section 3.3), we know that there exist $\kappa_1$ and $\kappa_2 {\geq}0$ such that, for all $\varphi \in \S$,
\begin{align*}
\int_{\R^d} |f(\varphi(\r)) | \, \du\r \leq  \mu \|\varphi\|_1 + \frac{\sigma^2}{2}\|\varphi\|_2^2 + \kappa_1 \|\varphi\|_\epsilon^\epsilon + \kappa_2 \|\varphi\|_2^2.
\end{align*}
{As} $\|\varphi\|_p $ {is finite} for all $p>0$, {we conclude that} $F(\varphi)$ is well-defined {over} $\S$. {In addition, from} Proposition \ref{mainprop} (Section 3.3) we know that there exist $\nu_1$ and $\nu_2{\geq}0$ such that, for all $\varphi,\psi\in\S$,
$$|F(\varphi)-F(\psi)| \leq \nu_1 \sqrt{(\|\varphi\|_\epsilon^\epsilon+\|\psi\|_\epsilon^\epsilon)(\|\varphi-\psi\|_\epsilon^\epsilon)} + \nu_2 \sqrt{(\|\varphi\|_2^2+\|\psi\|_2^2)(\|\varphi-\psi\|_2^2)}.$$
Consequently, if $\varphi_n \rightarrow \varphi$ in $\S$, then $F(\varphi_n)\rightarrow F(\varphi)$ in $\mathbb{C}$. {This shows that} $\widehat{\mathscr{P}}(\varphi)=\exp(F(\varphi))$ is continuous {over} $\S${, which completes the proof by applying the Minlos-Bochner theorem}. 
\end{proof}

\begin{table} [tb]
\centering
\caption{Tempered innovation processes}\label{table:Distributions}
\bigskip
\begin{tabular}{c| c c }
  \hline
  \hline
	Distribution & L\'evy triplet & Generalized L\'evy exponent\\
	and parameters & $(\mu,\sigma^2,V(\du a))$ & $F(\varphi)$ \\
  	\hline
  	Gaussian & $(\mu,\sigma^2,0)$ & $\ju\mu \left( \int \varphi \right) - \frac{\sigma^2\lVert \varphi \rVert_2^2}{2}$ \\
	$(\mu,\sigma^2) \in \R\times\R_+$ &  & \\
	\\
	S$\alpha$S & $\left( 0,0,\frac{C_{\alpha,\gamma}}{|a|^{\alpha+1}}\du a\right)$ & $-\gamma^\alpha \lVert \varphi \rVert_\alpha^\alpha$ \\ 
	$(\alpha,\gamma) \in (0,2)\times\R_+$ &$V\in \mathscr{M}(\alpha^-,\alpha^+)$ & \\
	\\
	{Variance Gamma} & $\left( 0,0, \frac{\eu^{-\lambda |a|}}{|a|}\du a \right)$ & $\int_{\R^d} \log\left( \frac{\lambda^2}{\lambda^2 + \varphi(\r)^2} \right) \du \r$\\
	$\lambda \in \R_+$ & $V\in \mathscr{M}(\infty,0^+)$ & \\
	\\
	Poisson & $(0,0,\lambda P(\du a))$ & $- \ju \lambda \mu_1^0(P) \left( \int \varphi \right) + \lambda \int_{\R^d}\int_{\R \backslash \{0\}} \left(\eu^{\ju a\varphi(\r)} -1 \right) P(\du a)\mathrm{d}\r$ \\  	$\lambda>0$, $P$ probability measure & $V\in\mathscr{M}(0^+,0)$ &  \\
 \hline 
 \hline
\end{tabular}
\end{table}

The  restriction {$V\in \mathscr{M}(0^{+},2)$} in Theorem \ref{theoS} is extremely mild and plays no role in all cases of practical interest (Table \ref{table:Distributions}). Yet, it is possible to construct examples of L\'evy measures {$V \in \mathscr{M}(0,2) \setminus \mathscr{M}(0^{+},2)$} such as
\begin{equation}
V(\du a)=\frac{\du a}{|a|\log^2(2+|a|)}.
\end{equation}

We give in Table 2 the main examples of white L\'evy-Schwartz noises: Gaussian noises, symmetric $\alpha$-stable noises (noted {by} S$\alpha$S, see \cite{Taqqu1994}), {Variance Gamma} noise ({which includes the Laplace distribution and is linked} with TV-regularization \cite{BosKamNil}), and Poisson noises.

\subsection*{Link Between Innovation Processes on $\D'$ and $\S'$}

Let $f$ be a L\'evy exponent with a L\'evy-Schwartz measure. {According to Theorems \ref{theoGV} and \ref{theoS}, we can  define,  }
\begin{itemize}
\item a measure $\mathscr{P}_{\D'}$ on $\D'$ such that $\CF_{\D'}(\varphi) = \exp\left( \int f(\varphi(\r)) \du\r \right)$ {for $\varphi\in\D$, and}
\item a measure $\mathscr{P}_{\S'}$ on $\S'$ such that $\CF_{\S'}(\varphi) = \exp\left( \int f(\varphi(\r)) \du\r \right)$ {for $\varphi\in\S$}, respectively.
\end{itemize}
We discuss here the compatibility {of} the two measures. Let $\mathcal{N}=\D$ or $\S$. {First, we recall the method of constructing a} measure on the nuclear space $\mathcal{N}$. {For given $\boldsymbol{\varphi}=(\varphi_1,\cdots,\varphi_N) \in \mathcal{N}^N$ and a Borelian subset $B$ of $\R^N$, a cylindric set is defined as
\begin{align}
A^{\mathcal{N'}}_{\boldsymbol{\varphi},B}=\{u\in \mathcal{N}', (\langle u ,\varphi_1\rangle ,\cdots\langle u ,\varphi_N\rangle) \in B\}.
\end{align}
If $\mathcal{C}_{\mathcal{N'}}$ denotes the collection of all such cylindric sets, then, according to the Minlos-Bochner theorem, the  $\sigma$-field $\mathcal{A}_{\mathcal{N'}} = \sigma(\mathcal{C}_{\mathcal{N'}})$ generated by the cylindric sets properly specifies a probability measure on $\mathcal{N}'$. In Proposition \ref{propTribu}, we compare the $\sigma$-fields $\mathcal{A}_{\mathcal{S}'}$  and $\mathcal{A}_{\mathcal{D}'}$.
Note that it is  not obvious \emph{a priori} that the two $\sigma$-fields are closely related. The main difficulty is to see that the space $\mathcal{S}'$ itself is an element of $\mathcal{A}_{\mathcal{D}'}$. The result of Proposition \ref{propTribu} is necessary to be able to compare the two measures $\mathscr{P}_{\mathcal{S}'}$ and $\mathscr{P}_{\mathcal{D'}}$.

}
%

{\begin{prop} \label{propTribu}
We have the  relations
\begin{eqnarray}
\mathcal{A}_{\mathcal{S}'} & = & \{ A \cap \mathcal{S}' \ | \ A  \in \mathcal{A}_{\mathcal{D}'}  \} \\
& \subset  & \mathcal{A}_{\mathcal{D}'}.
\end{eqnarray}

\end{prop}

\begin{proof}  We decompose the proof in four steps.  \\

(1) We denote $\mathcal{C}_{\mathcal{N'}}^{\Omega}$ the collection of cylindric set $A^{\mathcal{N'}}_{\boldsymbol{\varphi},B}$ with $\Omega$ an open set of $\R^N$. We claim that $\sigma(\mathcal{C}_{\mathcal{N'}}^{\Omega}) = \mathcal{A}_{\mathcal{N}'}$. This result is obtained by a transfinite induction using the fact that the open sets generates the $\sigma$-field of the Borelian sets.\\

(2) We show that $\S' \in \mathcal{A}_{\D'}$. For $\alpha \in \mathbb{N}$, we consider (see Table 1) $N_\alpha (\varphi) = \displaystyle{\sum_{0\leq |\textbf{k}| \leq \alpha}} \lVert \partial^{\textbf{k}} \varphi \rVert_{\infty,\alpha}$ where $|\textbf{k}| = k_1+\cdots+k_d$ and $\partial^{\textbf{k}}\varphi = \frac{\partial^{\textbf{k}}}{\partial r_1^{k_1} \cdots \partial r_d^{k_d}} \varphi$. \\

A generalized function $u \in \D'$ is tempered iff. there exist $\alpha \in \mathbb{N}$ and $C>0$ such that, for all $\varphi \in \D$, $|\langle u , \varphi \rangle | \leq CN_\alpha (\varphi)$. Then, $u$ can be uniquely extended to a continuous linear form on $\S$. The space $\S'$ is identified as a subspace of $\D'$ \cite{Bony2001}. In addition, we know that $\S$ is separable: there exists a sequence $(\varphi_n)\in \S^{\mathbb{N}}$ that is dense in $\S$. Because $\D$ is dense in $\S$, we can also imposed that the $\varphi_n$ are in $\D$. Consequently, we have
\begin{eqnarray}
\S' &=& \displaystyle{\bigcup_{C\in \mathbb{N}}}\displaystyle{\bigcup_{\alpha\in \mathbb{N}}}\displaystyle{\bigcap_{n\in \mathbb{N}}} A_{\varphi_n,[-CN_\alpha{\varphi_n},CN_\alpha{\varphi_n}]} \in \mathcal{A_{\D'}}.
\end{eqnarray}

A direct consequence is that $\{ A \cap \mathcal{S}' \ | \ A  \in \mathcal{A}_{\mathcal{D}'}  \}\subset \mathcal{A}_{\D'}$. \\

(3) First, we remark that $\{ A \cap \mathcal{S}' \ | \ A  \in \mathcal{A}_{\mathcal{D}'}  \} \cap \S' $ is a $\sigma$-field on $\S'$ (as a restriction of a $\sigma$-field on $\D'$) containing $\mathcal{C}_{\D'}^{\Omega} \cap \S'$ and then $\sigma(\mathcal{C}_{\D'}^{\Omega} \cap \S')$. Consequently, it is enough to show that $\mathcal{C}_{\S'}^{\Omega} \subset \sigma(\mathcal{C}_{\D'}^{\Omega} \cap \S')$. Let us fix $\psi_1,\cdots,\psi_N \in \S$ and $\Omega$ an open set of $\R^N$.
Let $(\varphi_{n,k})_{n=1,\cdots,N, \ k\in \mathbb{N}}$ be $N$ sequences of functions in $\D$ converging in $\S$ to $\psi_n$ for all $n \in \{1,\cdots,N\}$. 
Because $\Omega$ is open, for all $u\in \S'$, $(\langle u, \psi_1 \rangle, ..., \langle u, \psi_N \rangle) \in \Omega$ iff. $(\langle u, \varphi_{1,k} \rangle, ..., \langle u, \psi_{N,k} \rangle) \in \Omega$ for $k$ large enough. Moreover, because $\D \subset \S$, we have $A_{\boldsymbol{\psi},\Omega}^{\D'} \cup \S' = A_{\boldsymbol{\psi},\Omega}^{\S'}$. Thus,
\begin{eqnarray*}
A^{\S'}_{\boldsymbol{\psi},B} & = & \bigcup_{p\in \mathbb{N}} \bigcap_{k\geq p}  \left( A_{\boldsymbol{\varphi_k},B}^{\S'} \right)\\
& = &  \bigcup_{p\in \mathbb{N}} \bigcap_{k\geq p}   \left( A_{\boldsymbol{\varphi_k},B}^{\D'} \cap \S' \right) \\
& \in &  \sigma\left(\mathcal{C}_{\D'}^{\Omega} \cap \S' \right).
\end{eqnarray*}
As a consequence, $\mathcal{A}_{\S'} \subset\{ A \cap \mathcal{S}' \ | \ A  \in \mathcal{A}_{\mathcal{D}'}  \}$.  \\

(4) For the other inclusion, we first notice that $\{ A \cap \mathcal{S}' \ | \ A  \in \mathcal{A}_{\mathcal{D}'}  \}= \sigma (\mathcal{C}_{\D'}^\Omega \cap \S')$ (the restriction of the generator family generates the restrictive $\sigma$-field). Thus, we just need to prove that $A_{\boldsymbol{\psi},\Omega}^{\D'} \cup \S' \in \mathcal{A}_{\S'}$ for all $\psi \in \D^N$ and $\Omega$ an open set of $\R^N$, which is obvious because, as we said, $A_{\boldsymbol{\psi},\Omega}^{\D'} \cup \S' = A_{\boldsymbol{\psi},\Omega}^{\S'}$. Consequently,  $ \{ A \cap \mathcal{S}' \ | \ A  \in \mathcal{A}_{\mathcal{D}'}  \} \subset \mathcal{A}_{\S'}$. With (3), we obtain that $\{ A \cap \mathcal{S}' \ | \ A  \in \mathcal{A}_{\mathcal{D}'}  \} =   \mathcal{A}_{\mathcal{S}'}$.
\end{proof}
}

We now focus on measures {over $\mathcal{A}_{\D'}$ and $\mathcal{A}_{\S'}$ that define innovation processes.}

\begin{theo} \label{theoDS}
Let $f$ be a L\'evy exponent with L\'evy measure $V \in \mathscr{M}(0^+,2)$. {By} $\mathscr{P}_{\D'}$ and $\mathscr{P}_{\S'}$ {we denote} the measures on $\D'$ and $\S'${, respectively, that are defined by the} characteristic functional  $\exp \left( \int_{\R^d} f(\varphi (\r ))d\r \right)$ {(over $\D$ and $\S$, respectively)}. The two measures are compatible  in the sense that
\begin{equation}
\forall A \in \mathcal{A}_{\mathcal{S}'}, \ \mathscr{P}_{\D'}(A)=\mathscr{P}_{\S'}(A).
\end{equation}
In particular, $\mathscr{P}_{\D'}(\S')=1$ and $\mathscr{P}_{\D'}(\D' \backslash \S') = 0$.
\end{theo}

\begin{proof}
From $\mathscr{P}_{\S'}$, we define a new measure on $\D'$ by $\mathscr{P}(A) = \mathscr{P}_{\S'} (A\cap \S')$ for $A\in\mathcal{A}_{\D'}$. We {claim} that $\mathscr{P}=\mathscr{P}_{\D'}$. For $\varphi \in \D$, we have
\begin{eqnarray}
\widehat{\mathscr{P}}(\varphi) &=& \int_{\D'} \eu^{\ju\langle u,\varphi\rangle} \du\mathscr{P}(u) \nonumber \\
		&=& \int_{\S'} \eu^{\ju\langle u,\varphi \rangle} \du\mathscr{P}(u) \nonumber \\
		&=& \int_{\S'} \eu^{\ju\langle u,\varphi\rangle } \du\mathscr{P}_{\S'}(u) \label{eq:line3}  \\ 
		&=& \exp\left( \int_{\R^d} f(\varphi(\r)) \du\r \right)  \label{eq:line4} \\
		&=& \widehat{\mathscr{P}}_{\D'}(\varphi). \nonumber
\end{eqnarray}
We used that $\mathscr{P}(\D'\backslash \S') = 0$ in \eqref{eq:line3} and that $\mathscr{P}$ restricted to $\S'$ coincides with $\mathscr{P}_{\S'}$ in \eqref{eq:line4}. The Minlos-Bochner theorem ensures that $\mathscr{P} = \mathscr{P}_{\D'}$. Fix $A \in \S'$. According to Proposition \ref{propTribu}, $A \in \mathcal{A}_{\D'}$ and $\mathscr{P}_{\D'} (A)$ is well-defined. Consequently, we have {that} $\mathscr{P}_{\D'} (A) = \mathscr{P} (A) = \mathscr{P}_{\S'}(A\cap \S') =\mathscr{P}_{\S'}(A)$. For $A=\S'$, we obtain $\mathscr{P}_{\D'}{(\S')} =1$ and, consequently, $\mathscr{P}_{\D'} (\D' \backslash \S') =0$.
\end{proof}

The essential fact is that the theory of Gelfand already defines probability measures concentrated on the tempered generalized functions. 

\subsection{Sparse Processes}\label{subsec:SparseProcess}

{In the remainder of the paper, we restrict our attention to $\mathcal{N} =\S$. In Section \ref{subsec:LevyNoise}, we defined stochastic processes on $\S'$, which are \emph{bona fide} innovation processes. }
According to \cite{Unser_etal2011}, these innovation processes are split in{to} two categories: (i) white Gaussian noises corresponding to {zero L\'evy measure and} the L\'evy exponent $f(\omega) = \ju\mu \omega - \frac{\sigma^2}{2} \omega^2$, (ii) {non-Gaussian white noises with non zero L\'evy measures, which are referred to as} sparse innovation processes. The reason is that all non-Gaussian infinite-divisible {distributions} are necessarily {more compressible} than Gaussian{s} \cite{Unser_etal2011}.

Our {next} goal is to define processes $s$ such that $\mathrm{L}s=w$ is an innovation process. {With the rationale as above, the processes leading to} non-Gaussian innovation{s} $w$ will be called sparse.  The  fact that the {linear} operator $\mathrm{L}$ whitens  $s$ implies that  there exists a deterministic operator  $\mathrm{L}^{-1}$  that induces the dependency structure of the process $s$. If we follow the formal equalities $ \langle s, \varphi \rangle = \langle \mathrm{L}^{-1} w , \varphi \rangle = \langle w , \mathrm{L}^{*-1} \varphi \rangle$, we then  interpret the model $\mathrm{L}s=w$ as
\begin{equation} \label{eq:fL}      
\CF_s(\varphi) = \exp\left(\int_{\mathbb{R}^{d}}f(\mathrm{L}^{*-1}\varphi(\mathbf{r}))\du\mathbf{r}\right).
\end{equation}

 {However, we do not know \emph{a priori} if   \eqref{eq:fL} defines a valid characteristic functional. This is especially true if  (a) the operator $\mathrm{L}^{*-1}$ is continuous from $\mathcal{S}$ to some function space $\mathcal{T}$ (not necessarily nuclear) and if (b) the functional $\psi \mapsto \exp\left(\int_{\mathbb{R}^{d}}f(\psi(\mathbf{r}))d\mathbf{r}\right)$ is well-defined and continuous over   $\mathcal{T}$.} More precisely, we are {examining the} compatibility {of the L\'evy exponent} $f$  {with the linear operator} $\mathrm{L}$ to define  a characteristic functional.  Concretely,  we  are concerned with the function spaces $\mathcal{T}=\S$, $\mathcal{R}$, $L^{p}$-spaces, {and} intersection{s} of $L^{p}$-spaces. {We state in Theorem \ref{maintheo}   a sufficient compatibility condition. For the sake of generality, the operator $\mathrm{L}^{*-1}$ is replaced with the generic  operator $\mathrm{T}$.} 

\begin{theo} \label{maintheo} (Compatibility conditions) Let $f$ be a L\'evy exponent with triplet $(\mu,\sigma^2,V)$ and {let} $\mathrm{T}$ {be} a linear operator from $\S$ to $\S'$. Suppose we have $0<p_{\min}\leq p_{\max}\leq 2$ {and}
\begin{itemize} 
\item $V \in \mathscr{M}(p_{\min},p_{\max})$,
\item {$p_{\min} \leq 1$}, if $\mu \neq 0$ or $V$ non-symmetric,
\item {$p_{\max}=2$}, if $\sigma^2 \neq 0$, {and}
\item $\mathrm{T}$ is a linear and continuous operator from $\mathcal{S}$ to $L^{p_{\min}}\cap L^{p_{\max}}$.
\end{itemize}
Then, there exists a unique probability measure $\mathscr{P}_s$ on $\mathcal{S}'$
with
\begin{equation}
\widehat{\mathscr{P}_s}(\varphi)=\exp\left(\int_{\mathbb{R}^{d}}f(\mathrm{T}\varphi(\mathbf{r}))\du\mathbf{r}\right).
\end{equation}
\end{theo}

\begin{proof} 
We apply the Minlos-Bochner theorem to the functional {$\widehat{ \mathscr{P} }(\varphi) = \exp\left(\int_{\mathbb{R}^{d}}f(\mathrm{T}\varphi(\mathbf{r}))\du\mathbf{r}\right)$}. It is normalized because {the linearity of $\mathrm{T}$   implies} $f(\mathrm{T}\{0\})=f(0)=0$. {The linearity of $\mathrm{T}$ also enables us to conclude the positive-definiteness of the functional from Proposition \ref{positdefin}, given that it is} well-defined and continuous over $\S$. The continuity is {established by applying the bounds in Section \ref{subsec:Continuity}} on the generalized L\'evy exponent $F(\varphi) =\int_{\mathbb{R}^d} f(\varphi(\r))\du\r$. {In particular}, Corollary \ref{corgconti} implies {the existence of} $\kappa_1$ and $\kappa_2 \geq 0$ such that, for all $\varphi \in \S$,
$$\int_{\R^d} |f(T\varphi(\r)) | \du\r \leq  \mu \|T\varphi\|_1 + \frac{\sigma^2}{2}\|T\varphi\|_2^2 + \kappa_1 \|T\varphi\|_{p_{\min}}^{p_{\min}} + \kappa_2 \|T\varphi\|_{p_{\max}}^{p_{\max}},$$
and the assumptions on $f$ and $\mathrm{T}$ ensure that the integral is finite. Th{us}, $\widehat{\mathscr{P}}$ is well-defined. Under the assumption{s} of {T}heorem {\ref{maintheo}}, we can also apply Proposition \ref{mainprop} and find $\nu_1$ and $\nu_2{\geq}0$ such that, for all $\varphi, \psi \in \S$,
\begin{align*}
|F(\mathrm{T}\varphi)-F(\mathrm{T}\psi)|  \leq  \phantom{+}& \nu_1 \sqrt{(\|\mathrm{T}\varphi\|_{p_{\min}}^{p_{\min}}+\|\mathrm{T}\psi\|_{p_{\min}}^{p_{\min}})\|\mathrm{T}\varphi - \mathrm{T}\psi\|_{p_{\min}}^{p_{\min}}} \\
+& \nu_2 \sqrt{(\|\mathrm{T}\varphi\|_{p_{\max}}^{p_{\max}}+\|\mathrm{T}\psi\|_{p_{\max}}^{p_{\max}})\|\mathrm{T}\varphi - \mathrm{T}\psi\|_{p_{\max}}^{p_{\max}}}.
\end{align*}
{Now, if} $\varphi_n \rightarrow \varphi$ in $\S${, then}, $\mathrm{T}\varphi_n \rightarrow \mathrm{T}\varphi$ in $L^{p_{\min}}\cap L^{p_{\max}}$ (continuity of $\mathrm{T}$) and $F(\mathrm{T}\varphi_n)\rightarrow F(\mathrm{T}\varphi)$ in $\mathbb{C}$. Hence, $\widehat{\mathscr{P}}(\varphi)$ is continuous over $\S$. According to   Theorem \ref{theoMB}, there exists a unique $s$ such that $\widehat{\mathscr{P}_s} = \widehat{\mathscr{P}}$.

\end{proof}

\paragraph{Comparison with $p$-admissiblity.} 

Theorem \ref{maintheo} gives a compatibility condition between $f$ and $\mathrm{L}$. Another condition, called $p$-admissibility, {was introduced} in \cite{Unser_etal2011}.  A L\'evy exponent $f$ is said {to be} $p$-admissible if $|f(\omega)|+|\omega f'(\omega)| \leq C|\omega|^p$ for all $\omega \in \R$, where $1\leq p <+\infty$ and $C$ is a positive constant. {Although $p$-admissibility is sufficient in many practical cases, we argue that it is generally more restrictive than the assumptions in Theorem \ref{maintheo}}.
\begin{itemize}
	\item {The} $p$-admissibility {condition is restricted} to $p\geq 1$ {and}   requires the differentiability of the L\'evy exponent. The most natural sufficient condition to assure  differentiability is that $\mu_1(V)<+\infty$ ({or} $\mu_1^\infty(V)<+\infty$ {when} $V$ is symmetric). {In contrast, Theorem \ref{maintheo} does not impose the differentiability constraint and includes scenarios with $p<1$. }
	
	\item {The notion of $\mathscr{M}(p,q)$ introduced in Definition \ref{LSmeasure} distinguishes the limitations imposed by the L\'evy measure $V$ at $a\rightarrow 0$ and $a\rightarrow \infty$. As a result, Theorem \ref{maintheo} allows for a richer family of L\'evy exponents $f$. For instance, suppose }
that $f=f_\alpha+f_\beta$ is the sum of two S$\alpha$S L\'evy exponents with $\alpha < \beta$. Then{, although $f_\alpha$ and $f_\beta$ can be $\alpha$-admissible and $\beta$-admissible, respectively,} $f$ is not $p$-admissible for any $p>0${. It is not hard to check that $f$ is covered by Theorem \ref{maintheo}.} 
	
	 \item The {assumptions of} Theorem \ref{maintheo} can {also} be slightly restrictive. The S$\alpha$S case is {a} generic example. 
	 We denote $V_\alpha$ the L\'evy measure of the S$\alpha$S L\'evy exponent $f_\alpha$. Then, because $\mu_{\alpha}^{\infty}(V_\alpha) = \mu_{\alpha}^0(V_\alpha) = + \infty$, the
{Theorem \ref{maintheo} only allows for} {$V_\alpha \in \mathscr{M}(\alpha^- ,\alpha^+) = \bigcup_{\epsilon >0} \mathscr{M} (\alpha - \epsilon, \alpha + \epsilon)$}, but the condition $\varphi \in L^\alpha$ is clearly sufficient (and necessary) in practice. 
However, we know that $f_\alpha(\omega)=-|\omega|^\alpha$ is $\alpha$-admissible. 
\end{itemize}

\subsection{Continuity of {C}haracteristic {F}unctional{s}}\label{subsec:Continuity}

This section is devoted to {the derivation of} bounds on the generalized L\'evy exponent{s to conclude the continuity results required in} Theorems \ref{theoS} and \ref{maintheo}. We first introduce some notations and useful inequalities. 

\begin{defn}
Let $p>0$, $x,y\in\mathbb{R}$, and $f,g\in L^{p}(\mathbb{R}^{d})$. We define
\begin{eqnarray}
h_{p}(x,y) &=& \sqrt{\left(|x|^{p}+|y|^{p}\right)|x-y|^{p}}, \\
H_{p}(f,g) &=& \sqrt{\left(\|f\|_{p}^{p}+\|g\|_{p}^{p}\right)\|f-g\|_{p}^{p}}.
\end{eqnarray}
\end{defn}

\begin{lemma}\label{lemmaArash}
{Let $p,q>0$.}
\begin{enumerate}[(i)]
\item \label{arashxy1} For all $x,y\in\mathbb{R}$ we have {that} 
\begin{eqnarray} 
|x-y|^{p}& \leq & \max\left(1,2^{\frac{p-1}{2}}\right)h_{p}(x,y), \\
|x^{2}-y^{2}|^{p/2} &\leq & \max\left(1,2^{\frac{p-1}{2}}\right)h_{p}(x,y).
\end{eqnarray}

\item \label{arashfg} For $f,g\in L^{p}(\mathbb{R}^{d})$, we have  {that}
\begin{eqnarray} 
\int_{\mathbb{R}^{d}}h_{p}(f(\mathbf{r}),g(\mathbf{r}))\du\mathbf{r}\leq H_{p}(f,g) .
\end{eqnarray}

\item \label{arashconvex}  For $f,g \in L^p \cap L^q$ and $\lambda \in [0,1]$,
\begin{eqnarray} 
H_{\lambda p + (1-\lambda) q } (f,g) \leq \sqrt{\lambda} H_p(f,g) + \sqrt{1-\lambda} H_q(f,g) .
\end{eqnarray}
\end{enumerate}
\end{lemma}

\begin{proof} {For} $p\geq1$, {it follows from Jensen's inequality that} $|x{\pm}y|^{p}\leq2^{p-1}\left(|x|^{p}+|y|^{p}\right)$. Moreover, for $0<p<1$, we have {that} $|x{\pm}y|^{p}\leq|x|^{p}+|y|^{p}$. Consequently,
\begin{eqnarray*}
|x-y|^{p} & = & \sqrt{|x-y|^{p}}\sqrt{|x-y|^{p}}\\
 & \leq & \max(1,2^{\frac{p-1}{2}})h_{p}(x,y)
\end{eqnarray*}
and
\begin{eqnarray*}
|x^{2}-y^{2}|^{p/2} & = & \sqrt{|x+y|^{p}}\sqrt{|x-y|^{p}}\\
 & \leq & \max(1,2^{\frac{p-1}{2}})h_{p}(x,y).
\end{eqnarray*}
Let now $f,g\in L^{p}(\mathbb{R}^{d})$. By invoking the Cauchy-Schwartz {inequality}, we {can verify that}
\begin{eqnarray*}
\int_{\mathbb{R}^{d}}h_{p}(f(\mathbf{r}),g(\mathbf{r}))\du\mathbf{r} & = & \int_{\mathbb{R}^{d}}\sqrt{|f(\mathbf{r})|^{p}+|g(\mathbf{r})|^{p}}\sqrt{|f(\mathbf{r})-g(\mathbf{r})|^{p}}\du\mathbf{r}\\
 & \leq & \sqrt{\int_{\mathbb{R}^{d}}\left(|f(\mathbf{r})|^{p}+|g(\mathbf{r})|^{p}\right)\du\mathbf{r}} ~~ \sqrt{\int_{\mathbb{R}^{d}}|f(\mathbf{r})-g(\mathbf{r})|^{p}\du\mathbf{r}}\\
 & = & H_{p}(f,g).
\end{eqnarray*}

To prove (\ref{arashconvex}),  we define $F(\textbf{r}_1,\textbf{r}_2) = f(\textbf{r}_1)(g-f)(\textbf{r}_2)$ and $G(\textbf{r}_1,\textbf{r}_2)=g(\textbf{r}_1)(f-g)(\textbf{r}_2)$. As a consequence, $H_p(f,g)=\sqrt{\|F\|_p^p+\|G\|_p^p}$.
We now write {that}
\begin{eqnarray*}
H_{\lambda p +(1-\lambda) q} (f,g) & = & \sqrt{\|F\|_{\lambda p +(1-\lambda) q}^{\lambda p +(1-\lambda) q} +\|G\|_{\lambda p +(1-\lambda) q}^{\lambda p +(1-\lambda) q}} \\
& \leq & \sqrt{\lambda \|F\|_p^p + (1-\lambda) \|F\|_q^q +\lambda \|G\|_p^p + (1-\lambda) \|G\|_q^q } \\
& & [\text{using the convexity of }p\mapsto a^p \text{ for }a\geq 0] \\
& \leq & \sqrt{\lambda \left(\|F\|_p^p+\|G\|^p_p\right)} + \sqrt{(1-\lambda)\left(\|F\|_q^q+\|G\|_q^q\right)} \\
& & [\text{{using} the concavity of }\sqrt{.} ]\\
& = &  \sqrt{\lambda} H_p(f,g) + \sqrt{1-\lambda} H_q(f,g).
\end{eqnarray*}
\end{proof}

The {key} step {towards obtaining} continuity bounds is to control the non-Gaussian part $g$ of the L\'evy exponent. 

\begin{lemma}  [Control of $g(\omega)$] \label{mainlemma} {Let $V$ be a L\'evy measure, and define $\mathcal{A}_{sym}=\emptyset$ if $V$ is symmetric and $\mathcal{A}_{sym}=\{1\}$ otherwise. For some $0<p\leq q\leq2$ let $\mathcal{A}=\{p,q\}\cup\mathcal{A}_{sym}$ and set $p_{\min}=\min\mathcal{A}$ and $p_{\max}=\max\mathcal{A}$. Then, if $V\in\mathscr{M}(p_{\min},p_{\max})$, for the function}
\begin{equation}
g(\omega)=\int_{\mathbb{R}\backslash\{0\}}\left(\eu^{\ju\omega a}-1-\ju\omega a \1_{|a|<1}\right)V(\du a),
\end{equation}
there exist constants $\kappa_{1}$ and $\kappa_{2}\geq0$ such that, for all $(\omega_{1},\omega_{2})\in\mathbb{R}^{2}$, 
\begin{equation}
|g(\omega_{2})-g(\omega_{1})|\leq\kappa_{1}h_{p_{\min}}(\omega_{1},\omega_{2})+\kappa_{2}h_{p_{\max}}(\omega_{1},\omega_{2}). 
\end{equation}
\end{lemma}

\begin{proof} 
We decompose $\left( g(\omega_{1})-g(\omega_{2}) \right)$ in{to} 4 parts {as}
\begin{align*}
g(\omega_{1})-g(\omega_{2})  = &\phantom{+} \int_{|a|<1}\big(\cos(a\omega_{1})-\cos(a\omega_{2})\big)V(\du a)\\
 & +  \int_{|a|\geq1}\big(\cos(a\omega_{1})-\cos(a\omega_{2})\big)V(\du a)\\
 &  + \ju\int_{|a|<1}\big(\sin(a\omega_{1})-\sin(a\omega_{2})-a(\omega_{1}-\omega_{2})\big)V(\du a)\\
 & +  \ju\int_{|a|\geq1}\big(\sin(a\omega_{1})-\sin(a\omega_{2})\big)V(\du a)\\
  = &\phantom{+}  g_{_{\Re,0}}(\omega_{1,2})+g_{_{\Re,\infty}}(\omega_{1,2})+g_{_{\Im,0}}(\omega_{1,2})+g_{_{\Im,\infty}}(\omega_{1,2}).
\end{align*}

{To simplify the notations} we introduce $\Delta=\frac{a}{2}(\omega_{1}-\omega_{2})$ and $\Sigma=\frac{a}{2}(\omega_{1}+\omega_{2})${. We can write that}
\begin{eqnarray*}
\left\{\begin{array}{lll}
\cos(a\omega_{1})-\cos(a\omega_{2}) & = & -2\sin(\Delta)\sin(\Sigma),\\
\sin(a\omega_{1})-\sin(a\omega_{2}) & = & 2\sin(\Delta)\cos(\Sigma).
\end{array}\right.
\end{eqnarray*}

\begin{enumerate} 

\item We start with $g_{_{\Re,0}}$ and use the fact that $|\sin x|\leq \min(1,|x|)\leq |x|^{\frac{p_{\max}}{2}}$, because $p_{\max}\leq2$.
\begin{eqnarray*}
|g_{_{\Re,0}}(\omega_{1,2})| & \leq & 2\int_{|a|<1}|\sin(\Delta)\sin(\Sigma)|V(\du a)\\
 & \leq & 2\int_{|a|<1}|\Delta\Sigma|^{\frac{p_{\max}}{2}}V(\du a)\\
 & \text{=} & 2^{1-p_{\max}} \, \mu_{p_{\max}}^{0}(V) \, |\omega_{1}^{2}-\omega_{2}^{2}|^{\frac{p_{\max}}{2}}\\
 & \leq & \max\left(2^{1-p_{\max}},2^{\frac{1-p_{\max}}{2}}\right) \, \mu_{p_{\max}}^{0}(V) \, h_{p_{\max}}(\omega_{1},\omega_{2}),
\end{eqnarray*}
{where} we used part (\ref{arashxy1}) of {L}emma \ref{lemmaArash} {with} $p=p_{\max}$ {for the last inequality}.

\item For $g_{_{\Re,\infty}}$, we use  $|\sin x|\leq|x|^{\frac{p_{\min}}{2}}$
and {part (\ref{arashxy1}) of L}emma \ref{lemmaArash}  with $p=p_{\min}$ to obtain
\begin{eqnarray*}
|g_{_{\Re,\infty}}(\omega_{1,2})| & \leq & 2\int_{|a|\geq1}|\sin(\Delta)\sin(\Sigma)|V(\du a)\\
 & \leq & 2\int_{|a|\geq1}|\Delta\Sigma|^{\frac{p_{\min}}{2}}V(\du a)\\
 & \text{=} & 2^{1-p_{\min}} \, \mu_{p_{\min}}^{\infty}(V) \, |\omega_{1}^{2}-\omega_{2}^{2}|^{\frac{p_{\min}}{2}}\\
 & \leq & 2^{\frac{1-p_{\min}}{2}}\max\left(1,2^{\frac{1-p_{\min}}{2}}\right) \, \mu_{p_{\min}}^{\infty}(V) \, h_{p_{\min}}(\omega_{1},\omega_{2}).
\end{eqnarray*}

\item If $V$ is symmetric, then $g_{_{\Im,0}}=0$ and we do not need any bounds. For {a}symmetric cases, we know that $p_{\max}=\max(q,1)$. {Here, w}e use the inequality $|x-\sin x|\leq2|x|^{p_{\max}} \in [1,2]$ {for $1\leq p_{\max} \leq2$}. Indeed, $|x-\sin(x)|=|x|(1-\text{sinc}(x))\leq|x|\min(2,|x|)\leq2 |x| \times |x|^{p_{\max}-1}$. {By recalling} $\cos x= \left( 1-2\sin^{2}\left(\frac{x}{2}\right) \right)$, we have {that}
\begin{eqnarray*}
|g_{_{\Im,0}}(\omega_{1,2})| & = & 2\left\lvert \int_{|a|<1}\left(\sin(\Delta)\cos(\Sigma)-\Delta\right)V(\du a)\right\rvert \\
 & = & 2\left\lvert \int_{|a|<1}\left(\Delta-\sin(\Delta)+2\sin(\Delta)\sin^{2}\left(\frac{\Sigma}{2}\right)\right)V(\du a)\right\rvert \\
 & \leq & 2\int_{|a|<1}\left(2|\Delta|^{p_{\max}}+2|\Delta|^{\frac{p_{\max}}{2}}|\Sigma/2|^{\frac{p_{\max}}{2}}\right)V(\du a)\\
 & = & 4\mu_{p_{\max}}^{0}(V)\left(|\omega_{1}-\omega_{2}|^{p_{\max}}+|\omega_{1}^{2}-\omega_{2}^{2}|^{\frac{p_{\max}}{2}}\right)\\
 & \leq & 2^{\frac{p_{\max}+5}{2}} \, \mu_{p_{\max}}^{0}(V) \, h_{p_{\max}}(\omega_{1},\omega_{2}),
\end{eqnarray*}
where we used {part (\ref{arashxy1}) of L}emma \ref{lemmaArash}  for $p=p_{\max} \geq1$.

\item Again, if $V$ is symmetric, $g_{_{\Im,\infty}}=0$. {The construction of $\mathcal{A}$ f}or {a}symmetric cases {implies that} $p_{\min}=\min(p,1)$. By using the inequality $|\sin x|\leq|x|^{p_{\min}}$ and {part (\ref{arashxy1}) of  L}emma \ref{lemmaArash}, we get {that}
\begin{eqnarray*}
|g_{_{\Im,\infty}}(\omega_{1,2})| & \leq & 2\int_{|a|\geq1}|\sin(\Delta)\cos(\Sigma)|V(\du a)\\
 & \leq & 2\int_{|a|\geq1}|\Delta|^{p_{\min}}V(\du a)\\
 & = & 2\mu_{p_{\min}}^{\infty}(V) \, |\omega_{1}-\omega_{2}|^{p_{\min}}\\
 & \leq & 2\mu_{p_{\min}}^{\infty}(V) \, h_{p_{\min}}(\omega_{1},\omega_{2}).
\end{eqnarray*}

\end{enumerate}
We now just have to sum the four bounds to get the result.

\end{proof}

\begin{cor}\label{corgconti}
Under the same assumptions of Lemma \ref{mainlemma}, 
$$|g(\omega)| \leq\kappa_{1}|\omega|^{p_{\min}}+\kappa_{2}|\omega|^{p_{\max}}.$$
\end{cor} 

\begin{proof}
The result is obvious by setting $(\omega_{1},\omega_{2})=(\omega,0)$ in Lemma \ref{mainlemma}.
\end{proof}

{W}e now focus on {bounding the} generalized L\'evy exponent $G(\varphi)=\int_{\mathbb{R}^{d}}g(\varphi(\mathbf{r}))\du \mathbf{r} $ with {no} Gaussian part. 

\begin{lemma}[Control of $G(\varphi)$]\label{controlG} 
{Let $V$ be a L\'evy measure and define $\mathcal{A}_{sym}=\emptyset$ if $V$ is symmetric and $\mathcal{A}_{sym}=\{1\}$ otherwise. For some $0<p\leq q\leq2$ let $\mathcal{A}=\{p,q\}\cup\mathcal{A}_{sym}$ and set $p_{\min}=\min\mathcal{A}$ and $p_{\max}=\max\mathcal{A}$. Then, if $V\in\mathscr{M}(p_{\min},p_{\max})$, the functional}
\begin{equation}
 G(\varphi)=\int_{\mathbb{R}^{d}}g(\varphi(\mathbf{r}))\du\mathbf{r} \end{equation}
 is {well-}defined on $L^{p_{\min}}\cap L^{p_{\max}}$ and there exist $\kappa_{1},\kappa_{2}\geq0$ such that, for {all} $\varphi$ and $\psi\in L^{p_{\min}}\cap L^{p_{\max}}$,
\begin{equation}
|G(\varphi)-G(\text{\ensuremath{\psi})}|\le\kappa_{1}H_{p_{\min}}(\varphi,\psi)+\kappa_{2}H_{p_{\max}}(\varphi,\psi).
\end{equation}

\end{lemma}

\begin{proof} {w}e use Corollary \ref{corgconti} to prove that $G$ is well-defined. Indeed,
\[
\int_{\mathbb{R}^{d}}|g(\varphi(\mathbf{r}))|\du\mathbf{r}\leq\kappa_{1}\|\varphi\|_{p_{\min}}^{p_{\min}}+\kappa_{2}\|\varphi\|_{p_{\max}}^{p_{\max}}<+\infty,
\]
which proves that $g(\varphi) \in L^1$. Then, we apply {L}emmas \ref{lemmaArash} and \ref{mainlemma} 
 to {conclude that}
\begin{eqnarray*}
|G(\varphi)-G(\text{\ensuremath{\psi})}| & \leq & \int_{\mathbb{R}^{d}}|g(\varphi(\mathbf{r}))-g(\psi(\mathbf{r}))|\du\mathbf{r}\\
 & \leq & \kappa_{1}\int_{\mathbb{R}^{d}}h_{p_{\max}}(\varphi(\mathbf{r}),\psi(\mathbf{r}))\du\mathbf{r} +\kappa_{2}\int_{\mathbb{R}^{d}}h_{p_{\max}}(\varphi(\mathbf{r}),\psi(\mathbf{r}))\du\mathbf{r}\\
 & \leq & \kappa_{1}H_{p_{\max}}(\varphi,\psi)+\kappa_{2}H_{p_{\max}}(\varphi,\psi).
\end{eqnarray*}
\end{proof}

By {including} the Gaussian part, we now give the {continuity} condition {in its general form} for the characteristic functional of innovation processes.

\begin{prop}[Continuity of the characteristic functional]\label{mainprop}   
Let $f$ be a L\'evy exponent with triplet $(\mu,\sigma^{2},V)$ {and let }
$0<p\leq q\leq2$. We define
\begin{itemize}
\item \emph{$\mathcal{A}_{sym}=\emptyset$ if $V$ is symmetric and $\{1\}$
otherwise, $ $}
\item $\mathcal{A}_{1}=\emptyset$ if $\mu=0$ and $\{1\}$ otherwise,\emph{$ $}
\item $\mathcal{A}_{2}=\emptyset$ if $\sigma^2=0$ and $\{2\}$ otherwise,
\item $\mathcal{A}=\{p,q\}\cup\mathcal{A}_{sym}\cup\mathcal{A}_{1}\cup\mathcal{A}_{2}$,
\item $p_{\min}=\min\mathcal{A}$ and $p_{\max}=\max\mathcal{A}$.
\end{itemize}
{If $V\in\mathscr{M}(p_{\min},p_{\max})$, t}hen, the generalized L\'evy exponent $ F(\varphi)=\int_{\mathbb{R}^{d}}f(\varphi(\mathbf{r}))\du\mathbf{r} $ is {well-}defined on $L^{p_{\min}}\cap L^{p_{\max}}$ and  there exist $\nu_{1},\nu_{2}\geq0$ such that, for {all} $\varphi$ and $\psi\in L^{p_{\min}}\cap L^{p_{\max}}$, we have {that}
\begin{equation}
|F(\varphi)-F(\text{\ensuremath{\psi})}|\le{\nu}_{1}H_{p_{\min}}(\varphi,\psi)+{\nu}_{2}H_{p_{\max}}(\varphi,\psi).
\end{equation} 
This implies that  $F(\varphi)$ is continuous {over} $L^{p_{\min}}\cap L^{p_{\max}}$.
 \end{prop}

\begin{proof}  We use {L}emma \ref{controlG} to {justify that} $G(\varphi)=\int_{\mathbb{R}^{d}}g(\varphi(\mathbf{r}))\du\mathbf{r}$
is {well-}defined {over} $L^{p_{\min}}\cap L^{p_{\max}}$ and there exist $\kappa_{1},\kappa_{2}>0$
such that
$$|G(\varphi)-G(\text{\ensuremath{\psi})}|\le\kappa_{1}H_{p_{\min}}(\varphi,\psi)+\kappa_{2}H_{p_{\max}}(\varphi,\psi), $$
{where $\varphi,\psi\in L^{p_{\min}}\cap L^{p_{\max}}$. }
{Also, the inclusion of $\mathcal{A}_1$ and $\mathcal{A}_2$ in the definition of $\mathcal{A}$ (and therefore $p_{\min}, \ p_{\max}$), imposes bounds on the Gaussian part as} 
\begin{eqnarray*}
|F(\varphi)-F(\psi)| & = & \left\lvert \ju\mu \left( \int(\varphi-\psi) \right) -\frac{\sigma^{2}}{2}\|\varphi-\psi\|_{2}^{2}+G(\varphi)-G(\psi) \right\rvert \\
 & \leq & |\mu|H_{1}(\varphi,\psi)+\frac{\sigma^{2}}{2}H_{2}(\varphi,\psi)+\kappa_{1}H_{p_{\min}}(\varphi,\psi)+\kappa_{2}H_{p_{\max}}(\varphi,\psi)\\
 & \leq & \nu_{1}H_{p_{\min}}(\varphi,\psi)+\nu_{2}H_{p_{\max}}(\varphi,\psi).
 \end{eqnarray*}
{To validate the last inequality, note that, if  $\mu\neq 0$ ($\sigma\neq 0$), then $1\in[p_{\min},p_{\max}]$ ($2\in[p_{\min},p_{\max}]$). Hence, using Lemma \ref{lemmaArash}, $H_{1}(\varphi,\psi)$ ($H_{2}(\varphi,\psi)$) can be upper-bounded by a linear combination of $H_{p_{\min}}(\varphi,\psi)$ and $H_{p_{\max}}(\varphi,\psi)$.}

Finally, {to ensure the continuity of $F$}, we {point out the fact} that $H_p(\varphi_n,\varphi) \rightarrow 0$ if $\varphi_n \rightarrow \varphi$ in $L^p$.
\end{proof}

\section{Applications for Particular Classes of Operators}\label{sec:Applications}

We use  the previous results, mainly Theorem \ref{maintheo}, for particular classes of differential operators $\mathrm{L}$  (classical or fractional). {For each case, w}e summarize the hypotheses on $w$ required {in our results} to define the sparse process $s$ with $\mathrm{L}s=w$. 

We first {review the necessary steps for demonstrating} the existence of $s$. As mentioned {in Section \ref{subsec:SparseProcess}}, the interpretation of the innovation model $\mathrm{L}s=w$ is based on a characteristic functional of the form
\begin{equation} \label{eq:charac_func}
\CF_s(\varphi) = \exp\left(\int_{\mathbb{R}^{d}}f(\mathrm{L}^{*-1}\varphi(\mathbf{r}))\du\mathbf{r}\right).
\end{equation}
Let $\mathrm{L}$ be a linear operator defined on $\S$ (or a larger space) that {has the} adjoint operator $\mathrm{L}^*$ such that {its adjoint admits a linear} left inverse $\mathrm{L}^{*-1}{: \S \mapsto L^p \cap L^q}$. 
 Then, the characteristic functional of $\mathrm{L} s$ is $\CF_{\mathrm{L} s}(\varphi) = \exp\left(\int_{\mathbb{R}^{d}}f( \mathrm{L}^{*-1}  \mathrm{L} ^* \varphi(\mathbf{r}))\du\mathbf{r}\right) = \exp\left(\int_{\mathbb{R}^{d}}f(  \varphi(\mathbf{r}))\du\mathbf{r}\right)$ In other words, the operator $\mathrm{L} $ whitens the generalized process $s$.

\subsection{Self-Similar Sparse Processes} \label{subsec:SelfSimilar}
 
We are  interested in defining generalized stochastic processes {$s$} such that $\mathrm{L}s = w$, where  $\mathrm{L} = (-\Delta)^{\gamma/2}$ is the fractional Laplacian {operator} of order $\gamma>0$. Such processes  are called second-order self-similar because their correlation structure{s are invariant to} similarity transformations, due to the homogeneity of the fractional Laplacian operator. In the Gaussian case, self-similarity is intimately tied to fractional Brownian motion{s (fBm)}  \cite{Mandelbrot1968}. The link between innovation model{s arising from} fractional Laplacian {operators} and fBm is studied in \cite{TaftiLaplacian}. {This indicates implicitly} that such processes are special cases of the present framework.
{Here, by applying the fundamental results in \cite{Sun-frac}, w}e extend the definition of   {Gaussian} self-similar processes {to the larger class of self-similar processes with infinite-divisible distributions derived from L\'evy noises.} 
 
The fractional Laplacian {operator $(-\Delta )^{\gamma/2}$} is defined for $\varphi \in \S$ by $(-\Delta )^{\gamma/2} \varphi = \mathcal{F}^{-1} \left( \|.\|^{\gamma} \mathcal{F}\varphi \right)$ where $\mathcal{F}$ ($\mathcal{F}^{-1}$, respectively) is the Fourier transform (the inverse Fourier transform, respectively). It is linear, continuous from $\S$ to  $C^\infty$, rotational-, shift-,  scale-invariant, and  self-adjoint. {Thus, we need to find its linear left inverse operator(s).} 
For $0<\gamma<d$, its natural (left and right) inverse is the well-known Riesz potential $I_\gamma$. An extension for $\gamma > d$, $\gamma \notin \mathbb{N}$, called generalized Riesz potential, is {introduced in} \cite{Sun-frac}. The main results concerning such operators {can be found} in Theorem 1.1 {of} \cite{Sun-frac} {which is summarized in Theorem \ref{Sun1}}. 

\begin{theo}[Shift-invariant left inverse of $\FL$] \label{Sun1} 
Let $\gamma>0$ with $(\gamma - d) \notin \mathbb{N}$. The operator $I_\gamma$ with {frequency response} $\|\omega\|^{-\gamma}$ is the unique linear and  continuous operator from $\S$ to $\S'$ that is shift- and scale-invariant{. It is also} a left inverse of the fractional Laplacian {operator} on $S${, which implies that}
\begin{equation}
 {\forall\,\varphi\in\S,~~~ } I_\gamma \FL \varphi = \varphi.
 \end{equation}

\end{theo}

{In general, the output range of  operators $I_\gamma$ on $\varphi\in \S$ is not restricted to $\S$ or even to an $L^p$ space. However,  we can confine the range by limiting $\gamma$. More precisely, by considering a generalization of  the Hardy-Littlewood-Sobolev inequality for $1\leq p\leq +\infty$, we can show that}
	\begin{equation}
	I_\gamma (\S) \subset L^p \Leftrightarrow 0< \gamma < d(1-1/p). 
	\end{equation}
Consequently,
\begin{itemize} 
	\item {For} $\gamma <d$, {we have that} $I_\gamma(\S) \subset L^p$ iff. $p>\frac{d}{d-\gamma}>1$.
	\item {For} {$d<\gamma \not \in \mathbb{N}$, we have that} $I_\gamma(\S) \not\subset L^p$ for all $p\geq1$. 
\end{itemize}

Suppose now that $\gamma \notin \mathbb{N}$. Following \cite{Sun-frac}, we are able to define a correction of $I_\gamma$ from $\S$ to $L^p$ for all fixed $p\geq 1$ such that $\left( \gamma - d(1-1/p) \notin \mathbb{N} \right)$. We denote the set of forbidden values of $p$ {by  $A(d,\gamma)$, which is given by}
\begin{eqnarray}
	A(d,\gamma) =	\left\{ \begin{array}{ll}  \left\{ \frac{d}{d+1-\epsilon(\gamma) - k},\, k\in \{ 0,...,\lfloor \gamma \rfloor \} \right\} & \mbox{if } \gamma < d, \phantom{\bigg|}\\
	                  						  \left\{ \frac{d}{d+1-\epsilon(\gamma) - k},\, k\in \{ 0,...,d  \} \right\} & \mbox{if } \gamma > d 	. \phantom{\bigg|}
						\end{array}\right.
\end{eqnarray}

{The set} $A(d,\gamma)$ contains $k(d,\gamma)=\min( \lfloor \gamma \rfloor + 1, d)$ elements. {Therefore}, $[1,+\infty] \backslash A(d,\varphi)$ {is composed of} $k(d,\gamma) +1$ {intervals given by}
\begin{eqnarray}
C(d,\gamma,k)  = \left\{ \begin{array}{ll} 	\left[ 1 , \frac {d}{d- \epsilon(\gamma)} \right) & \text{if } k=0, \\
 								\left( \frac{d} {d - k +1 - \epsilon(\gamma) } , \frac{d}{d- k-\epsilon(\gamma) } \right) & \text{if } k\in \{ 1,...,k(d,\gamma)-1\}, \phantom{\bigg|}\\
								\left( \frac{d} {d-k+1 - \epsilon(\gamma) } , + \infty \right] & \text{if } k= k(d,\gamma). 
				\end{array}\right.
\end{eqnarray}

For instance, if $d=1$, then, $k(1,\gamma)=1$ {and} there is only one forbidden value, with $p=\frac{1}{1-\epsilon(\gamma)}$. {Also, the two intervals are} $C(1,\gamma,0)=\big[1,\frac{1}{1-\epsilon(\gamma)}\big)$ and $C(1,\gamma,1)=\big(\frac{1}{1-\epsilon(\gamma)},+\infty \big]$. {Similarly, f}or $d=2$, there are two {intervals} if $\gamma < 1$ and three otherwise.

We {modify} $I_\gamma$ {to guarantee some} $L^p$ stability, for $k\in \{0,...,k(d,\gamma) \}$, by
 \begin{eqnarray}
  \mathcal{F}( I_{\gamma,k}\{ \varphi \} ) ( \boldsymbol{\omega})= \|  \boldsymbol{\omega} \|^{-\gamma} \left( \mathcal{F}\varphi ( \boldsymbol{\omega})   - \sum_{|\textbf{j}| \leq \lfloor \gamma \rfloor - k } \frac {  \partial^{\textbf{j}}\mathcal{F}\varphi (\textbf{0})}   {\textbf{j}!}  \boldsymbol{\omega} ^{\textbf{j}} 
  \right)  .
  \end{eqnarray}

Note that such operators are {no longer} shift-invariant{. This can be cast as the cost of obtaining an $L^p$-stable operator.} 

\begin{prop}[$L^p$-stable left inverse of $\FL$] \label{d+1}  
Let $d \in \mathbb{N}^*$,  $\gamma \in (0,+\infty) \setminus (\mathbb{N} + d )$, and $k \in \{ 0,...,k(d,\gamma) \}$. 
\begin{itemize}
\item The operator $I_{\gamma,k}$ is  {continuous} linear and scale-invariant   from $S$ to $L^p$, for all $p \in C(d,\gamma,k)$. Moreover, it is a left inverse of the Laplacian operator $\FL$. 

\item For fixed $p\in C(d,\gamma,k)$, $I_{\gamma,k}$ is the unique linear {and} scale-invariant left inverse of $\FL$ from $\S$ to $L^p$.
\item If $p$ and $q$ are in distinct $C(d,\gamma,k)$ sets, then{ the Laplacian operator $(-\Delta)^{\gamma/2}$ has} no linear and scale-invariant left inverse from $\S$ to $L^p \cap L^q$.
\end{itemize}
\end{prop}

\begin{proof} The {first two claims} are direct rewritings of Theorem 1.2 in \cite{Sun-frac} {by noting} that, for $p\in [1,+\infty]\setminus A(\gamma)$, 
\begin{equation}
p \in C(d,\gamma,k) ~~ \Leftrightarrow ~~ \left\lfloor \gamma - d\left( 1 - \frac{1}{p} \right) \right\rfloor  = \lfloor \gamma \rfloor - k.
\end{equation}
{The last claim follows from the uniqueness property and states that} the conditions {for restricting the range to} $L^p$ and $L^q$ are incompatible. 
\end{proof}

We are now able to give admissibility conditions between $\gamma$ and a L\'evy exponent $f$ to define processes whitened by fractional Laplacian.

\begin{prop}\label{GammaV}
Let $\gamma \in (0,+\infty) \backslash \mathbb{N}$ and $f$ be a L\'evy exponent with triplet $(\mu, \sigma^2,V)$. {Define} $p_{\min}$ and $p_{\max}$ as in Theorem \ref{maintheo} and {let} $k$ {be} such that $p_{\min}$ and $p_{\max} \in C(d,\gamma,k)$. Then, there exists a generalized stochastic process $s$ with
 \begin{equation}
\widehat{\mathscr{P}} _s(\varphi) = \exp \left( \int_{\R^d} f(I_{\gamma,k} \varphi(\textbf{r})) \mathrm{d}\textbf{r} \right)
\end{equation}
{for $\varphi\in\S$. The process} $s$ is a broadly \emph{self-similar sparse process} (\emph{self-similar Gaussian process}, respectively) if $V{\not \equiv}0$ (if $V{\equiv}0$, respectively). 
\end{prop}

{Proposition \ref{GammaV} can be interpreted as follows: there exists a process $s$, such that} $\langle s, \varphi \rangle = \langle w , I_{\gamma,k} \varphi \rangle = \langle I_{\gamma,k}^* w , \varphi \rangle$. In other words, $(-\Delta)^{\gamma/2} s = w$. 

\begin{proof}
We apply Theorem \ref{maintheo} with $\mathrm{T}=I_{\gamma,k}$, which is continuous on $L^{p_{\min}} \cap L^{p_{\max}}$, according to Proposition \ref{d+1}.
\end{proof}

{We examine the construction of the innovation models in Proposition \ref{GammaV} for} L\'evy-Schwartz measures $V$ with finite first{-order} moment $\mu_1^\infty(V) < + \infty$.
\begin{itemize}
\item {\textbf{Gaussian case}.} Suppose that $\left( \gamma - \frac{d}{2} \right) \not\in \mathbb{N}$. Then, there exists $k=\lfloor d/2+1-\epsilon(\gamma)\rfloor$ such that $I_{\gamma,k}$ is continuous from $\S$ to $L^2${. Thus, according to Proposition \ref{GammaV}, the functional $\CF_s(\varphi) =\exp\left( -\frac{\sigma^2\lVert I_{\gamma,k} \varphi \rVert_2^2}{2} \right)$ defines a process $s$ which is} whitened by $\FL$. 

\item {\textbf{Laplace case}.} The L\'evy measure $V_L$ of {a} Laplacian law verifies $V_L \in \mathscr{M}(1,1)$. Let $p_{\min} = p_{\max} = 1$. Proposition \ref{GammaV} applies with  the operator $I_{\gamma,0}$ for $\gamma \notin \mathbb{N}$.

\item {\textbf{Compound Poisson case}.} Suppose that $V \in \mathscr{M}(1,0)$. Then, as in the Laplace case, the operator $I_{\gamma,0}$ is admissible for all $\gamma \notin \mathbb{N}$.

\item {\textbf{S$\alpha$S}.} Let $1\leq \alpha < 2$ and $\gamma \notin \mathbb{N}$ with $\left( \gamma - d\left( 1-1/\alpha\right) \right)  \notin \mathbb{N}$. Then, there exists $k= \lfloor d(1-1/\alpha) +1-\epsilon(\gamma) \rfloor$ such that $I_{\gamma,k}(\S) \subset L^\alpha$. According to Proposition \ref{GammaV}, there exists $s$ with $\CF_s(\varphi) = \exp ( - \lVert I_{\gamma,k} \varphi \rVert_\alpha^\alpha )$ and $\FL s = w$ with $w$ an $\alpha$-stable innovation process.
\end{itemize}

\begin{figure}[tb] 
  \centering
    \includegraphics[width=0.60 \textwidth]{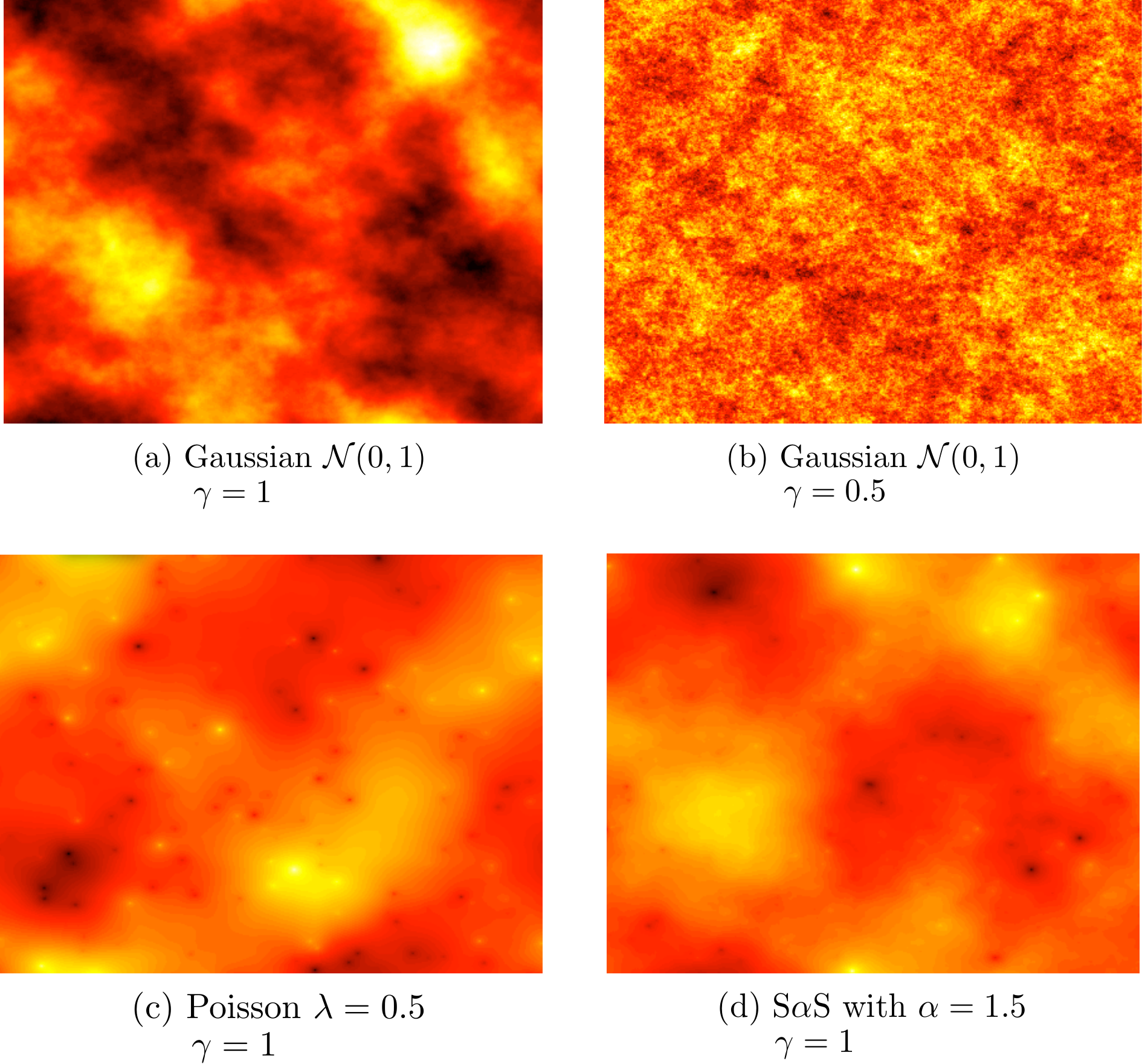}
    \caption{Self-similar processes}
\end{figure} 

We {depict} in Figure 2 some examples of self-similar processes in dimension $d=2$. Dark intensities correspond to the highest values of the simulated process, while  bright ones correspond to the smallest.

\subsection{Directional Sparse Processes} \label{subsec:Directional}

Our goal is to define directional stochastic processes on $\S'$ using oriented differential  operators. {This} consists {of} defining {proper} left inverse operators {for} derivative operators {of the form} $\Der_{\u} \varphi = \langle \bigtriangledown \varphi, \u \rangle = u_1 \Der_1 \varphi + \cdots + u_d \Der_d \varphi$, {where $\u$ stands for the direction. For this purpose, we extend the one-dimensional results of \cite{Unser_etal2011} to higher dimensions.} We start with {first-order operators}  $\mathrm{L}=  \Der_{\u}-\alpha \mathrm{Id}  $ with $\alpha \in \C$ and $\u\in\R^d$.   \\

We denote by $(\e_k)$ the canonical bas{is} of $\R^d$. For $\u \in \R^d \backslash \{\bold{0}\}$, $p_{\u^\perp} (\r) = \r - \frac{\langle \u,\r \rangle}{\lVert \u \rVert_2^2} \u$ is the orthogonal projection on $\u^\perp = \{ \v \, | \, \langle \u,\v\rangle =0 \}$. Recall that $\lVert \r \rVert_2^2 = \langle \u,\r \rangle^2 \lVert \u \rVert_2^2 + \lVert p_{\u^\perp} (\r) \rVert_2^2$. {Since} $\Der_{\u} = \lVert \u \rVert_2 \Der_{\u / \lVert \u \rVert_2}$, we assume now that $\lVert \u \rVert_2=1$, without loss of generality.

\subsubsection{Left Inverse Operators of $\mathrm{L}=\Der_{\u}-\alpha \mathrm{Id}$}

Let $\u \in \R^d$ with norm {$\|\u\|_2=1$. We separately investigate the cases of $\Re (\alpha)\neq 0$ and $\Re (\alpha)= 0$, as they result in stable and marginally stable left inverses, respectively.

\subsubsection*{Left inverse operators in the stable case}
Since the case of $\Re (\alpha)>0$ is very similar to $\Re (\alpha)<0$, we first study the  causal case. Therefore, we assume that $\Re (\alpha)<0$ and w}e define $\rho_{\u,\alpha}$ by
\begin{equation}
\langle \rho_{\u,\alpha} , \varphi \rangle =  \int_0^{+\infty} \eu^{\alpha t} \varphi(t\u)\du t .
\end{equation}
{We further define the operator} $I_{\u,\alpha}$ on $\S$ as $I_{\u,\alpha} = \rho_{\u,\alpha} * \varphi$.  In {one }dimension, $\rho_{\u,\alpha}$ is a causal exponential function {in the classical sense as was introduced} in \cite{Unser_etal2011}{. However,} for $d \geq 2$, it is a generalized function.

\begin{prop}\label{inversepartialdifferential}
The continuous operator $I_{\u,\alpha}$ is LSI {and} continuous from $\S$ to $\S${. Furthermore, it is} the inverse of the partial differential operator $(\Der_{\u} - \alpha \mathrm{Id})$ on $\S${, meaning that}
\begin{equation}
 I_{\u,\alpha}(\Der_{\u}-\alpha \mathrm{Id}) \varphi = (\Der_{\u} - \alpha \mathrm{Id}) I_{\u,\alpha}  \varphi = \varphi {,~~~ \forall\,\varphi\in\S}.
 \end{equation}
\end{prop}

\begin{proof} 
First, because $| \langle \rho_{\u,\alpha} ,\varphi \rangle | \leq \|\varphi\|_{\infty} \int_0^{+\infty} | \eu^{\alpha t}| \du t= \lVert \varphi \rVert_{\infty} / (-\Re(\alpha))$, we have {that} $\rho_{\u,\alpha} \in \S'${. This confirms that} $I_{\u,\alpha}$ is well-defined on $\S$. 

The derivative of $\rho_{\u,\alpha}$ in the sense of generalized function{s} is
\begin{eqnarray*}
\langle \Der_{\u} \rho_{\u,\alpha} , \varphi \rangle & = & - \langle \rho_{\u,\alpha} , \Der_{\u} \varphi \rangle \\
& = & - \int_0^{+\infty} \eu^{\alpha t} \{ \Der_{\u} \varphi \} (t\u) \du t \\
 & = & {- \left[ \varphi(t\u) \eu^{\alpha t } \right]_0^{+\infty} + \alpha \langle \rho_{\u,\alpha} , \varphi \rangle }\\
 &    & [\text{using an integration by parts} ]   \\
 & = & {\varphi(0) + \alpha  \langle \rho_{\u,\alpha} , \varphi \rangle, }
\end{eqnarray*}
 meaning that  $(\Der_{\u}  - \alpha \mathrm{Id})  \rho_{\u,\alpha} = \delta$. Consequently, in Fourier domain, we have {that}
\begin{eqnarray*}
1 &  = &  \widehat{\Der_{\u}\rho_{\u,\alpha}} (\boldsymbol{\omega}) - \alpha \widehat{\rho_{\u,\alpha}} (\boldsymbol{\omega})\\
 & = & (\ju \langle \boldsymbol{\omega}, \u \rangle - \alpha) \widehat{\rho_{\u,\alpha}} (\boldsymbol{\omega}).
\end{eqnarray*}
{This implies} $$\widehat{\rho_{\u,\alpha} * \varphi} (\boldsymbol{\omega})= \widehat{\rho_{\u,\alpha}}{(\boldsymbol{\omega})} \widehat{\varphi} (\boldsymbol{\omega})= \frac{\widehat{\varphi}(\boldsymbol{\omega})}{\ju  \langle \boldsymbol{\omega}, \u \rangle - \alpha} .$$
Consequently, $\widehat{\rho_{\u,\alpha} * \varphi}$ and $\rho_{\u,\alpha} * \varphi$ {belong to $\S$}. Moreover, we know that {the} LSI operator $\varphi \mapsto u * \varphi$ for $u\in \S'$, is continuous from $\S$ {in}to itself iff. $u\in\mathcal{O}'_C$ or, equivalently, iff. $\widehat{u} \in \mathcal{O}_M$, the space of slowly increasing and infinitely differentiable functions (see \cite{SchDistri} for more details). {Since in our case} $\widehat{\rho_{\u,\alpha}} \in \mathcal{O}_M$, {we conclude that} $I_{\u,\alpha}$ is continuous. 

{Finally, we can write that}
\begin{eqnarray*}
(\Der_{\u}-\alpha \mathrm{Id})  I_{\u,\alpha} \varphi & = & (\Der_{\u}-\alpha \mathrm{Id})  (\rho_{\u,\alpha} * \varphi) = ((\Der_{\u}-\alpha \mathrm{Id})  \rho_{\u,\alpha}) * \varphi = \delta * \varphi = \varphi, \\
I_{\u,\alpha} (\Der_{\u}-\alpha \mathrm{Id})  \varphi & = & \rho_{\u,\alpha} * ((\Der_{\u}-\alpha \mathrm{Id})  \varphi) = ((\Der_{\u}-\alpha \mathrm{Id})  \rho_{\u,\alpha}) * \varphi  = \varphi .
\end{eqnarray*}

\end{proof}

Following \cite{Unser_etal2011}, we {can transpose} this result for $\Re(\alpha) > 0$ (anti-causal case) {by} defining $\rho_{\u,\alpha} (\textbf{r}) = \rho_{\u, -\alpha} (-\textbf{r})$. {With this choice, we can show in a similar way that Proposition \ref{inversepartialdifferential} also holds for $\Re(\alpha)>0$.}

\subsubsection*{Left inverse operators in the marginally stable case}

Suppose now that $\alpha = \ju\omega_0$ is {purely} imaginary, with $\Re(\alpha) =0$. The natural candidate for $(\Der_{\u} - \ju\omega_0 \mathrm{Id})^{-1}$  is again the convolution {operator defined by the kernel $\rho_{\u,\ju\omega_0}$ where}
$$\langle \rho_{\u,\ju\omega_0} , \varphi \rangle = \int_0^{+\infty} \eu^{{\ju\omega_0} t} \varphi(t\u)\du t. $$
{In other words,}
\begin{equation}
I_{\u,\ju\omega_0}\varphi (\textbf{r}) = \eu^{\ju\omega_0 \langle \r, \u \rangle } \int_{-\infty}^{\langle \r, \u \rangle } \eu^{-\ju\omega_0 \tau} \varphi (p_{\u^\perp}(\r) + \tau \u) \du\tau.
\end{equation}
The adjoint of $I_{\u,\ju\omega_0}$ is given by
\begin{equation}
I_{\u,\ju\omega_0}^*\varphi (\textbf{r}) = (\rho_{\u,\ju\omega_0}(- \cdot  ) * \varphi) (\r) = \eu^{-\ju\omega_0 \langle \r, \u \rangle } \int_{\langle \r, \u \rangle }^{+\infty} \eu^{\ju\omega_0 \tau} \varphi  (p_{\u^\perp}(\r) + \tau \u)  \du\tau .
\end{equation}
{These $I_{\u,\ju\omega}^*$ operators are shift-invariant. However, their impulse responses are not rapidly decreasing (their Fourier transforms are not in $\mathcal{O}_M$). Consequently, they are not stable and cannot define valid characteristic functionals in \eqref{eq:charac_func}. Here, we propose  a modification inspired by \cite{Unser_etal2011} to overcome the instability problem.} We define
\begin{eqnarray}\label{eq:J_def}
J_{\u,\omega_0} \varphi (\textbf{r}) & = & I_{\u,\ju\omega_0} \varphi (\textbf{r}) - \{ I_{\u,\ju\omega_0} \varphi \} (p_{\u^\perp}(\r)) \eu^{\ju\omega_0 \langle \r, \u \rangle }  \label{eq:J-I} \\
& = & \eu^{\ju\omega_0 \langle \r, \u \rangle } \int_{0}^{\langle \r, \u \rangle } \eu^{-\ju\omega_0 \tau} \varphi (p_{\u^\perp}(\r) + \tau \u) \du\tau .
\end{eqnarray}
{The modified operator} $J_{\u,\omega_0}$ is  continuous from $\S$ to $C^\infty$ {and is a} right-inverse of $(\Der_{\u}-{\ju\omega_0 } \mathrm{Id})$. {Indeed, $ \left( I_{\u,j\omega_0} \varphi  - J_{\u,\omega_0} \varphi \right) (\r) = I_{\u,\ju\omega_0} \varphi (p_{\u^\perp}(\r)) \eu^{\ju\omega_0 \langle \r, \u \rangle } \in \mbox{Ker} (\Der_{\u}-\ju \omega_0 \mathrm{Id} )$.} 
 {We claim that its} adjoint is {given by}
\begin{eqnarray}\label{eq:Jadj_def}
J_{\u,\omega_0}^*\varphi (\textbf{r}) & = &I_{\u,\ju\omega_0}^* -  \1_{\langle \r , \u\rangle \leq 0} \eu^{-\ju\omega_0 \langle \r, \u \rangle } \int_{-\infty }^{+ \infty} \eu^{\ju\omega_0 \tau} \varphi  (p_{\u^\perp}(\r) + \tau \u)  \du\tau .  \label{eq:J*-I*}
\end{eqnarray}
We have now to show that 
\begin{equation} \label{eq:adjointrelation} 
 {\forall \varphi,\psi \in\S,~~~ } \langle J_{\u,\omega_0}\varphi \,,\, \psi\rangle = \langle \varphi \,,\, J_{\u,\omega_0}^{*}\psi\rangle.
\end{equation}
  We denote 
  \begin{eqnarray} 
  A(\varphi,\psi) &=&  \langle   I_{\u,\ju\omega_0} \varphi  (p_{\u^\perp}(\r)) \eu^{\ju\omega_0 \langle \r, \u \rangle }   , \ \psi(\r) \rangle \\
  B(\varphi,\psi) & =&  \langle \varphi(\r)\  , \  \1_{\langle \r , \u\rangle \leq 0} \eu^{-\ju\omega_0 \langle \r, \u \rangle }  \int_{-\infty }^{+ \infty} \eu^{\ju\omega_0 \tau} \psi  (p_{\u^\perp}(\r) + \tau \u)  \du\tau  \rangle.
  \end{eqnarray}
   The Equation \eqref{eq:adjointrelation} is equivalent (as we can see from  \eqref{eq:J-I} and  \eqref{eq:J*-I*})  to $A(\varphi,\psi) = B(\varphi,\psi)$. To show this, we denote $(\u_1 = \u, \u_2, \cdots, \u_d)$ an orthonormal basis of $\R^d$. Especially, we have  $p_{\u^\perp}(\r) = \langle \r,\u_2 \rangle \u_2 + \cdots +  \langle \r,\u_d \rangle \u_d$. Then,  
\begin{eqnarray*}
A(\varphi,\psi)  & = & \int_{\R^{d}}\left(  \eu^{\ju \omega_0 \langle \r,\u \rangle} \int_{\R} \1_{\tau \leq 0} \eu^{- \ju \omega_0 \tau} \psi ( p_{\u^\perp}(\r) + \tau \u ) \du\tau  \right)  \psi(\r)  \du \r  \\
&=& \int_{\R^{d+1}}   \1_{\tau\leq 0} \eu^{\ju \omega_0 r_1 } \eu^{- \ju \omega_0 \tau} \varphi (  \tau \u_1 + r_2 \u_2 + \cdots + r_d \u_d )     \psi( r_1 \u_1 + r_2 \u_2 + \cdots + r_d \u_d)  \du\tau \du r_1 \cdots \du r_d \\
& & [\text{because }(\u_1,\cdots,\u_d) \text{ is an orthonormal basis}] \\
&=& \int_{\R^{d+1}}   \1_{r_1'\leq 0} \eu^{\ju \omega_0 \tau' } \eu^{- \ju \omega_0 r_1'} \varphi (  r_1' \u_1 + r_2 \u_2 + \cdots + r_d \u_d )     \psi(\tau' \u_1 + r_2 \u_2 + \cdots + r_d \u_d)  \du\tau' \du r_1' \cdots \du r_d \\
& & [\text{using the change of variables }(r_1,\tau)=(\tau',r_1')] \\
&=& \int_{\R^{d}}  \varphi(\r) \left(  \1_{ \langle \r, \u \rangle \leq 0}   \eu^{- \ju \omega_0 \langle \r, \u \rangle } \int_{\R} \eu^{\ju \omega_0 \tau' }      \psi(\tau' \u +p_{\u^\perp}(\r))  \du\tau' \right)  \du\r \\
&=&  B(\varphi,\psi).
\end{eqnarray*}

{Denote the modulation operators by} $M_{\u,\omega_0} \varphi (\textbf{r}) = \eu^{\ju\omega_0 \langle \r, \u \rangle } \varphi(\textbf{r})${ Then it is not difficult to check that}
\begin{eqnarray}\label{eq:modulation}
I_{\u,j\omega_0} \varphi& = &  M_{\u,\omega_0} I_{\u,0} M_{\u,-\omega_0} \varphi , \nonumber\\
I_{\u,j\omega_0}^* \varphi& = &  M_{\u,-\omega_0} I_{\u,0}^* M_{\u,\omega_0} \varphi , \nonumber\\
J_{\u,\omega_0} \varphi& = &  M_{\u,\omega_0} J_{\u,0} M_{\u,-\omega_0} \varphi ,\nonumber\\
J_{\u,\omega_0}^*\varphi& = &  M_{\u,-\omega_0} J_{\u,0}^* M_{\u,\omega_0} \varphi .
\end{eqnarray}

Note that $J_{\u,\omega_0}$ preserves the regularity{, with} $J_{\u,\omega_0} \varphi \in C^\infty$. On the contrary, its adjoint creates discontinuit{ies along} the hyperplane $\langle \r , \u \rangle = 0${, while it} preserves the decay properties, as we can see in {P}roposition {\ref{controlDecrease}}.

\begin{prop}[{Properties of $J_{\u,\omega_0}^*$}]\label{controlDecrease} {The following properties hold for the adjoint operator $J_{\u,\omega_0}^*$ defined in \eqref{eq:Jadj_def}:}  
\begin{itemize}
\item The adjoint operator $J_{\u,\omega_0}^*$  is continuous from $L^{\infty,\alpha}$ to $L^{\infty,\alpha-1}$ for $\alpha>1$.
\item  The adjoint operator $J_{\u,\omega_0}^*$ is linear {and} continuous from $\mathcal{R}$ {in}to itself{, and it is a} left inverse of the operator {$(\Der_{\u}-\ju\omega_0 \mathrm{Id})^*$} on $\S$. 
\end{itemize}
\end{prop}

\begin{proof}
Because of the modulation {equalities in \eqref{eq:modulation}}, we only need to prove the {claims} for $\omega_0 = 0$. Let $\textbf{r} \in \{ \langle \r, \u \rangle \geq 0\}$ and $\alpha>1${. T}hen,
\begin{eqnarray*}
\lvert J_{\u,0}^* \varphi (\textbf{r}) \rvert & = & \left\lvert \int_{\langle \r, \u \rangle }^{+\infty} \varphi ( p_{\u^\perp}(\r) + \tau\u) \du\tau \right\rvert \\
	& \leq & \lVert \varphi \rVert_{\infty,\a} \int_{\langle \r, \u \rangle }^{+\infty} \frac{\du\tau}{1+ \lVert p_{\u^\perp}(\r) + \tau \u \rVert_2^{\alpha} } \\
	& = & \lVert \varphi \rVert_{\infty,\a} \int_{\langle \r, \u \rangle }^{+\infty} \frac{\du\tau}{1+ {\sqrt{ \lVert p_{\u^\perp}(\r) \rVert_2^2 + \tau^2  }}^{\alpha} } \\
	& &  \mbox{[using orthogonality of }  p_{\u^\perp}(\r) \mbox{ and } \tau \u] \\
	& \leq & \lVert \varphi \rVert_{\infty,\a} \int_{\langle \r, \u \rangle }^{+\infty} \frac{\du\tau}{1+ (2^{-1/2}(\lVert p_{\u^\perp}(\r)\rVert_2+\tau))^{\alpha} } \\
	& & \mbox{[using concavity of } \sqrt{.} ]\\
	& \leq & \lVert \varphi \rVert_{\infty,\a} \int_{(\langle \r, \u \rangle +\lVert p_{\u^\perp}(\r)\rVert_2)/\sqrt{2} }^{+\infty} \frac{\du\nu}{1+ \nu^{\alpha} } \\
	& \leq &  \frac{2^{3/2}\alpha}{\alpha - 1}  \lVert \varphi \rVert_{\infty,\a} \frac {1}{1+ 2^{-(\alpha-1)/2}(\langle \r, \u \rangle  +\lVert p_{\u^\perp}(\r)\rVert_2)^{\alpha-1}} \\
	& & \mbox{[using } \int_x^{+\infty} \frac{\du\nu}{1+\nu^{\alpha}} \leq \frac {2\alpha}{(\alpha -1)(1+x^{\alpha -1})} \mbox{ for } x\geq0 ]\\
	& \leq & \frac{2^{3/2}\alpha}{\alpha - 1}  \lVert \varphi \rVert_{\infty,\a} \frac {1}{1+ 2^{-(\alpha-1)/2}\lVert \r \rVert_2^{\alpha-1}}  \\
	&  & \mbox{[using } \langle \r, \u \rangle +\lVert p_{\u^\perp}(\r)\rVert_2 \geq \sqrt{\langle \r, \u \rangle^2  +\lVert p_{\u^\perp}(\r)\rVert_2^2} = \lVert \r \rVert_2].
\end{eqnarray*}
{Finally, w}e remark that $(1+\lVert \r \rVert_2^{\alpha-1}) \leq 2^{(\alpha-1)/2} (1+ 2^{-(\alpha-1)/2}\lVert \r \rVert_2^{\alpha-1})${, which yields}
$$ \left\lvert  J_{\u,0}^* \varphi (\textbf{r})(1+\lVert \r \rVert_2^{\alpha-1}) \right\rvert \leq C_{\alpha} \lVert \varphi \rVert_{\infty, \alpha} .$$

The same inequality holds for $\langle \r, \u \rangle < 0$ which ensures the continuity from  $L^{\infty,\alpha}$ to $L^{\infty,\alpha-1}$ for $\alpha>1$. Because $\mathcal{R} = \bigcap_{\a > 0} L^{\infty,\a}$, the previous bounds for all $\a>1$ imply  that {$J^{*}_{\u, 0}$} is continuous  from $\mathcal{R}$ {in}to itself. Moreover, {for $\langle \r, \u \rangle  \geq 0$ we have that}
\begin{eqnarray*} 
J_{\u,0}^* \Der_{\u}^* \varphi (\textbf{r}) & = & - \int_{\langle \r, \u \rangle }^{+\infty}  \Der_{\u} \varphi  \left(\textbf{r}+\left(\tau - \langle \r, \u \rangle \right)\u\right) \du\tau \\
& = & - \left[ \varphi\left(\textbf{r}+\left(\tau - \langle \r, \u \rangle \right)\u\right) \right]_{\tau=\langle \r, \u \rangle }^{\infty} \\
& = & \varphi(\textbf{r}).
\end{eqnarray*}
We get the same result for $\langle \r, \u \rangle  < 0${, which confirms that} $J_{\u,0}^*$ is a left inverse of $\Der_{\u}^*$.

\end{proof}

\subsubsection{Existence of Directional Sparse Processes}
\begin{figure}[tb]
  \centering
    \includegraphics[width=0.68 \textwidth]{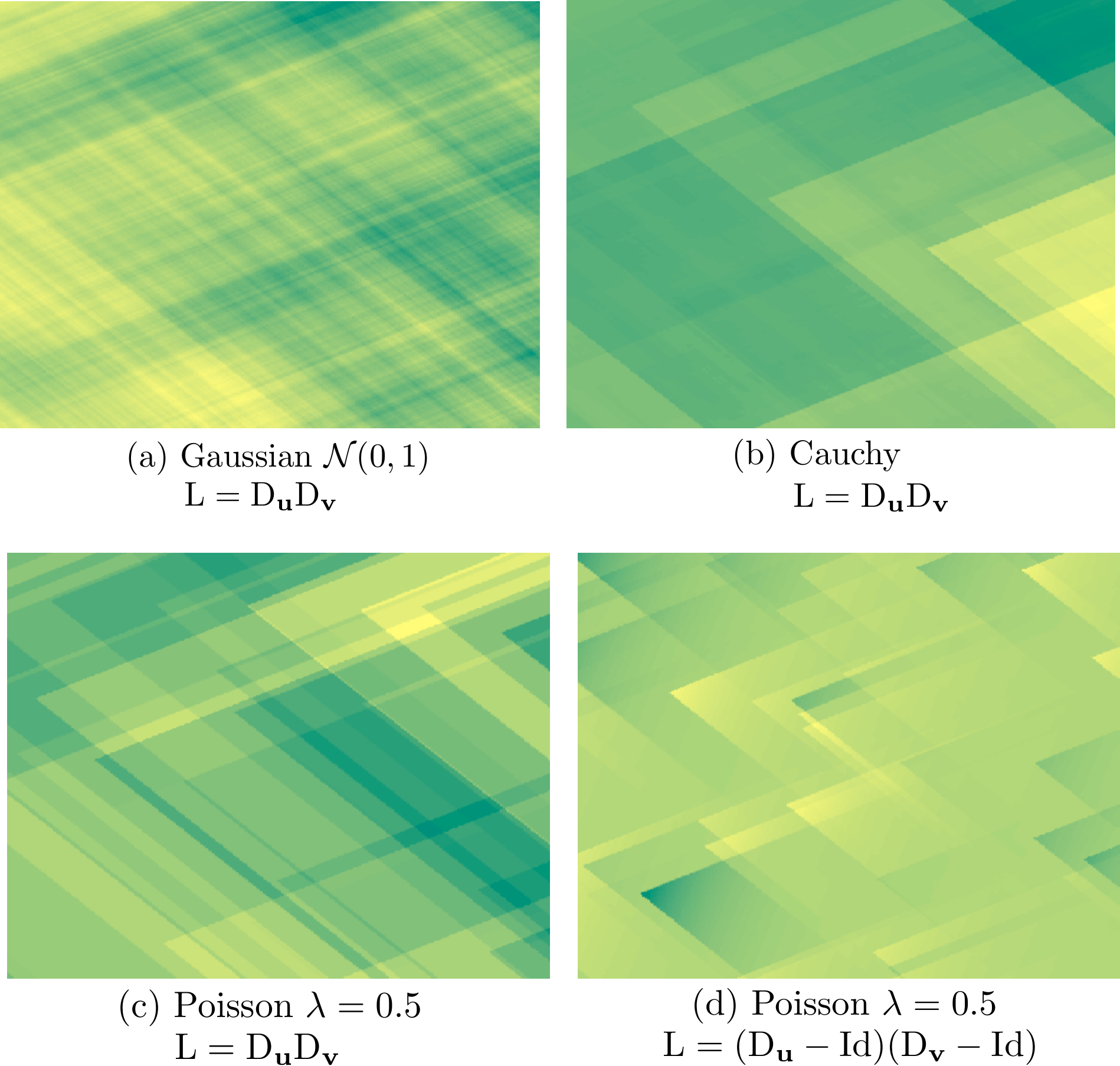}
    \caption{Directional Gaussian or sparse processes}
\end{figure}  
{Up to this point, we have covered the first-order directional differential operators. In general, a directional differential operators $\mathrm{L}$ can be decomposed as}
\begin{eqnarray}
\mathrm{L}& =  & (\Der_{\v_1}-\ju\omega_1 \mathrm{Id})\cdots (\Der_{\v_q}-\ju\omega_q \mathrm{Id})(\Der_{\u_1}-\alpha_1 \mathrm{Id})\cdots  (\Der_{\u_p}-\alpha_p \mathrm{Id}) \nonumber  \\
& = & \mathrm{L}_{\mathrm{critical}} \, \mathrm{L}_{\mathrm{LSI}}, 
\end{eqnarray}
{where $\Re(\alpha_i)\neq 0$. According to Proposition \ref{inversepartialdifferential}, each of the factors in $\mathrm{L}_{\mathrm{LSI}}$ has an $\S$-continuous adjoint left inverse. By composing these operators, we can define a continuous adjoint left inverse $\mathrm{L}_{\mathrm{LSI}}^{*-1}$ for $\mathrm{L}_{\mathrm{LSI}}$. Similarly, the results of Proposition \ref{controlDecrease} can be employed to form a continuous operator $\mathrm{L}_{\mathrm{critical}}^{*-1}$ from $\S$ to $\mathcal{R}$. Since the constituents of $\mathrm{L}_{\mathrm{critical}}^{*-1}$ are not shift-invariant, different composition orders may result in different operators. However, all of them are valid adjoint left inverse operators for $\mathrm{L}_{\mathrm{critical}}$.}

 Finally, $\mathrm{L}^{*-1} = \mathrm{L}_{\mathrm{critical}}^{*-1} \mathrm{L}_{\mathrm{LSI}}^{*-1}$ is a left inverse operator of $\mathrm{L}^* = \mathrm{L}_{\mathrm{LSI}}^* \mathrm{L}_{\mathrm{critical}}^*$, continuous from $\S$ to $\mathcal{R}$, linear, but not shift-invariant in general. Next, we define generalized stochastic processes based on such $\mathrm{L}$.

\begin{prop}  \label{existsMondrian}
Let $f$ be a L\'evy exponent with $V\in\mathscr{M}(0^{+},2)$, let $\mathrm{L}$ {be} a directional differential operator{, and let $\mathrm{L}^{*-1}$ stand for its adjoint left inverse.}
Then, there exists a generalized  stochastic process $s$ on $\S$ such that
\begin{equation}
\widehat{\mathscr{P}}_s(\varphi) = \exp \left(  \int_{\R^d} f(\mathrm{L}^{*-1} \varphi (\textbf{r})) \du\textbf{r} \right) .
\end{equation}
{The resulting process} $s$ is called a \emph{directional sparse process} (a \emph{directional Gaussian process, respectively}) if $V{\not\equiv}0$ (if $V\equiv 0$, respectively).
\end{prop}

\begin{proof}
Let $0<\epsilon\leq 1$ such that $V \in \mathscr{M}(\epsilon,2)$. {As mentioned earlier,} $\mathrm{L}^{*-1}$ is continuous from $\S$ to $\mathcal{R}$ {and, therefore,} from $\S$ to $L^\epsilon \cap L^2$. We can {now} apply Theorem \ref{maintheo} with $p_{\min}=\epsilon$ and $p_{\max} = 2$ {to complete the proof}. 
\end{proof}

{In summary, for all directional differential operators $\mathrm{L}$, we can define the process $s=\mathrm{L}^{-1}w$} if $V$ is a L\'evy-Schwartz measure. For instance, we can define {the} classical one-dimensional L\'evy processes (with the point of view of generalized stochastic process{es}) with $\mathrm{L}=\mathrm{D}$ as in \cite{Unser_etal2011}. We can also define  the {$d$-dimensional} Mondrian process with $\mathrm{L}=\Der_1\cdots \Der_d$ and $V\in\mathscr{M}(0^+,{0})$ which corresponds to a Poisson innovation process (see  Table \ref{table:Distributions}), as was done in \cite{TaftiPoisson} for $d=2$. \\

Let $d=2$. We consider in Figure 3 the case $\mathrm{L}= (\mathrm{D}_{\u} - \alpha \mathrm{Id})(\mathrm{D}_{\textbf{v}} - \beta \mathrm{Id})$ for some real numbers $\alpha,\beta$ and vectors $\u = (2,1)$ and $\textbf{v}=(2,-1)$. {Dark and bright colors indicate large and small values in the simulated realizations, respectively.} 
Note that the {first three} processes are non-stationary {due} to the non shift-invariance of the left inverse of $\mathrm{D}_{\u}\mathrm{D}_{\textbf{v}} $.

\section*{Acknowledgements}
The authors are grateful to R. Dalang for the fruitful discussions we had and V. Uhlmann for her helpful suggestions during the writing. The research leading to these results has received funding from the European Research Council under the European Union's Seventh Framework Programme (FP7/2007-2013) / ERC grant agreement n¡ 267439.

\bibliographystyle{alpha} 
\bibliography{refs}

\end{document}